\documentclass[nosumlimits,twoside]{amsart}
\usepackage{amsfonts, amsmath, amssymb}
\usepackage[bookmarksnumbered,plainpages,hypertex]{hyperref}
\usepackage{graphicx}
\usepackage{float}
\usepackage{srcltx}
\usepackage[all]{xy}

\newtheorem{theorem}{\sc Theorem}[section]
\newtheorem{proposition}[theorem]{\sc Proposition}
\newtheorem{notation}[theorem]{\sc Notation}

\newtheorem{lemma}[theorem]{\sc Lemma}
\newtheorem{corollary}[theorem]{\sc Corollary}
\theoremstyle{definition}
\newtheorem{definition}[theorem]{\sc Definition}
\newtheorem{definitions}[theorem]{\sc Definitions}
\newtheorem{example}[theorem]{\sc Example}

\theoremstyle{remark}
\newtheorem{remark}[theorem]{\sc Remark}

\newcounter{maint}

\def\yd{_H^H\mathcal{YD}}
\def\c{\mathfrak{c}}
\allowdisplaybreaks

\usepackage[normalem]{ulem}

\begin{document}
\title{Quantum Lines for Dual Quasi-Bialgebras }

\author{Alessandro Ardizzoni}
\address{\parbox[b]{\linewidth}{University of Turin, Department of Mathematics ``G. Peano'', via
Carlo Alberto 10, I-10123 Torino, Italy}}
\email{alessandro.ardizzoni@unito.it}
\urladdr{www.unito.it/persone/alessandro.ardizzoni}

\author{Margaret Beattie}
\address{Department of Mathematics and Computer Science, Mount Allison
University, Sackville, NB, Canada, \indent E4L 1E6}
\email{mbeattie@mta.ca}
\urladdr{http://www.mta.ca/~mbeattie/}

\author{Claudia Menini}
\address{\parbox[b]{\linewidth}{Department of Mathematics and Computer Science, University of Ferrara, Via Machiavelli
35, Ferrara, I-44121, Italy}}
\email{men@unife.it}
\urladdr{web.unife.it/utenti/claudia.menini}

\subjclass[2000]{Primary 16W30;  Secondary 16S40}

\thanks{This paper was written while the first and the third authors were members of GNSAGA.
The first author was partially supported by the research grant ``Progetti di
Eccellenza 2011/2012'' from the ``Fondazione Cassa di Risparmio di Padova e
Rovigo''. The second author was supported by an NSERC Discovery Grant.   Her stay, as a visiting
professor at University of Ferrara in 2011, was supported by INdAM}

\begin{abstract}
In this paper, the theory to construct quantum lines for general dual quasi-bialgebras is developed followed by some specific examples where the dual quasi-bialgebras are pointed with cyclic group of points.
\end{abstract}

\keywords{ Dual quasi-bialgebras; quantum lines; bosonizations}

\maketitle
\tableofcontents

\section{Introduction}

For $H$ a bialgebra over a field $\Bbbk$ and $R$ a bialgebra in the category of Yetter-Drinfeld modules $\yd$,
 the Radford biproduct or bosonization $R \#H$ is a well-known construction giving a new bialgebra.
 Similarly, if $H$ is a dual quasi-bialgebra and $R$ is a bialgebra in a suitably defined category of Yetter-Drinfeld modules
 over $H$, then by \cite{Ardi-Pava2}, a new dual quasi-bialgebra $R\# H$ can be defined, called the bosonization of $R$ and $H$.

\par Given a bialgebra $H$, however, finding $R$ is nontrivial.  For $H = \Bbbk G$, a group algebra,
 finding $R$ is the key to constructing pointed bialgebras with $G$ as the group of points and finding finite dimensional $R$ is
 crucial to the classification of pointed Hopf algebras of finite dimension.   If $R$ is finite dimensional and is generated
  as an algebra by a one-dimensional vector space in $\yd$, then $R$ is called a quantum line.  For various
 $H$ semisimple of even dimension, not necessarily  group algebras,  the question of the existence of a quantum line
 for $H$ was completely settled in \cite{CDMM-QuantumLines} as well as the question of liftings of the bosonizations.

 \par In this paper, we adapt the methods and language of \cite{CDMM-QuantumLines} to dual quasi-bialgebras. We find necessary and sufficient conditions
 for quantum lines to exist for a given dual quasi-bialgebra $H$ and we compute several examples for the dual quasi-bialgebra
 $(H,\omega)$ where $H$ is
 the group algebra of a cyclic group of any order and $\omega$ is a $3$-cocycle. Finally we give an example of the
 existence of a quantum line for a bosonization $R \# \Bbbk C_n$ where $n$ is an even integer.

\par The duals of our examples will be quasi-bialgebras
 as studied in  the papers of Angiono \cite{angiono}, Gelaki \cite{Gelaki}, and
 Etingof and Gelaki \cite{EGe codim 2}\cite{EGe rad gr}\cite{EGe liftings}.

\par The first section of this paper is used for notation and some preliminary material.
 In the second section, we define quasi-Yetter-Drinfeld data for dual quasi-bialgebras, and then in the next section we construct
 quantum lines.  Section \ref{sec: qYD for bosonizations} discusses conditions to construct a quantum line for a bosonization and
 then the last section gives examples of these constructions.  The examples are based on knowledge of the dual
 quasi-bialgebra $\Bbbk C_n$ with Drinfeld reassociator given by a nontrivial $3$-cocycle.

\section{Preliminaries}

Throughout we work over  $\Bbbk$, an algebraically closed field of characteristic zero. The
tensor product over $\Bbbk $ will be denoted by $\otimes $. Vector spaces,
algebras and coalgebras are all understood to be over $\Bbbk $  and all maps are understood to be
$\Bbbk$-linear. The usual twist map from the tensor space $V \otimes W$ to $W \otimes V$ will
be denoted $\tau$, i.e., $\tau( v \otimes w) = w \otimes v$.  The multiplicative group
of nonzero elements of $\Bbbk$ is denoted by $\Bbbk ^\times$.
\par For any
coalgebra $C$ and algebra $A$,  $\ast $ will denote the convolution product in $
\mathrm{Hom}\left( C,A\right) $.
Composition of functions may be written as concatenation if the emphasis of the symbol $\circ$ is not required.
The tensor product of a map with itself will often be written exponentially, i.e., we will write
$\phi^{\otimes 3}$ to denote $\phi \otimes \phi \otimes \phi$. Similarly $H \otimes H$ is denoted $H^{\otimes 2}$, etc.

\par The group algebra over a group $G$ will be written $\Bbbk G$. The set of grouplike elements
of a coalgebra $C$ will be denoted $G(C)$ and the subcoalgebra generated by $G(C)$ will be denoted $\Bbbk G(C)$.

\par  We  make
the convention that an empty product, for example, a product of the form $\prod\limits_{1\leq j\leq a}$ with $a <1$, is defined to be $1$.

\subsection{Definitions}

 A \emph{coalgebra with
multiplication and unit} is a datum $\left( H,\Delta ,\varepsilon
,m,u\right) $ where $\left( H,\Delta ,\varepsilon \right) $ is a coalgebra, $%
m:H\otimes H\rightarrow H$ is a coalgebra homomorphism called multiplication
and $u:\Bbbk \rightarrow H$ is a coalgebra homomorphism called unit \cite[dual to page 368]{Kassel-Quantum}.
 Denote $u(1_\Bbbk)$ by $1_H$.\\

 For $H$ a coalgebra with multiplication and unit,
 a convolution invertible map $\omega :H^{\otimes 3} \rightarrow  \Bbbk $
is called a $3$-cocycle if and only if
\begin{equation}
\left( \varepsilon \otimes \omega \right) \ast \omega \left( H\otimes
m\otimes H\right) \ast \left( \omega \otimes \varepsilon \right) =\omega
\left( H\otimes H\otimes m\right) \ast \omega \left( m\otimes H\otimes
H\right),  \label{eq:3-cocycle}
\end{equation}
and we say that a cocycle $\omega$ is  unitary or normalized   if for all $h, h^\prime \in H$,
\begin{equation}
\omega (h\otimes 1_{H}\otimes h^{\prime }),  \text{ or equivalently either }\omega(1 \otimes h \otimes h^\prime)
\text{ or } \omega(h \otimes h^\prime \otimes 1), \text{ is }
\varepsilon(h)\varepsilon(h^\prime).
    \label{eq:quasi-unitairitycocycle}
\end{equation}

If $H,L$ are coalgebras with multiplication and unit,  a coalgebra map
$\phi:L \rightarrow H$ is a morphism of coalgebras
with multiplication and unit, if
\begin{equation*}
m_{H }(\phi \otimes \phi )=\phi m_L,\qquad \phi u_L=u_{H
} .
\end{equation*}%
If $\phi:L \rightarrow H$ is a morphism of coalgebras
with multiplication and unit and $\omega$ is a (normalized) $3$-cocycle for $H$, then $\omega   \circ \phi^{ \otimes 3}$  is a  (normalized) $3$-cocycle for $L$. \\

For $H$ a coalgebra with unit, a convolution invertible map
$v: H^{ \otimes 2} \rightarrow \Bbbk$ such that $v(1 \otimes h) = v(h \otimes 1) = \varepsilon(h)$ for all $h \in H$,
i.e., $v$ is unitary or normalized, is called a \textit{gauge transformation}.\\

Note that for $H$ a cocommutative bialgebra    and $v: H \otimes H \rightarrow \Bbbk$ a (normalized) convolution invertible map
 then
the map $\partial^2v: H \otimes H \otimes H \rightarrow \Bbbk$, defined by
\begin{equation*}
\partial^2v: = (\varepsilon \otimes v )\ast v^{-1}(m \otimes H) \ast v(H \otimes m) \ast (v^{-1} \otimes \varepsilon),
\end{equation*}
is a  (normalized)  cocycle  called a  (normalized)  coboundary.
Conversely, if $v$ is convolution invertible and $\partial^2v$ is normalized
 then $v(1 \otimes h) = v(h \otimes 1) = \varepsilon(h) v(1 \otimes 1) $ and so $ v(1 \otimes 1)^{-1}v$ is normalized. (See also Lemma \ref{lm: normalize to a gauge trans}.)

Now   we  define    the objects of interest in this paper.

\begin{definition} A \emph{dual quasi-bialgebra} is a datum $\left( H,\Delta ,\varepsilon
,m,u,\omega \right) $ where $\left( H,\Delta ,\varepsilon ,m,u\right) $ is a
coalgebra with multiplication and unit and $\omega :H\otimes H\otimes
H\rightarrow \Bbbk $ is a normalized $3$-cocycle such that
\begin{eqnarray}
{}[m\left( H\otimes m\right) ]\ast (u\omega ) &=&(u\omega )\ast \lbrack
m\left( m\otimes H\right) ];  \label{eq:quasi-associativity} \\
m(1_{H}\otimes h) &=&h=m(h\otimes 1_{H}),\text{ for all }h\in H.
\label{eq:unitarity}
\end{eqnarray}
 \end{definition}

Unless it is needed to emphasis the structure of the coalgebra $H$ with multiplication and unit, we will
 write $(H, \omega)$ for a dual quasi-bialgebra.  The map $\omega $ is called the (Drinfeld) \emph{reassociator} of the dual
quasi-bialgebra. Note that if $H$ is cocommutative, then $(H,\omega)$ has
associative multiplication for every reassociator $\omega$.\\

\vspace{1mm}

Following \cite[Section 2]{Schauenburg-Quotients}, we say that $\Phi
:\left( H,   \omega \right) \rightarrow \left(
H^{\prime },
 \omega ^{\prime }\right) $ is a \emph{morphism of dual quasi-bialgebras}
if  $\Phi :  H  \rightarrow
H^{\prime }  $
 is a
morphism of coalgebras with multiplication and unit and
$
 \omega ^{\prime } \circ \Phi^{\otimes 3}   =\omega .
 $

A bijective morphism of dual quasi-bialgebras is an isomorphism.\\
\vspace{1mm}

A \emph{dual quasi-subbialgebra} of a dual quasi-bialgebra {   $\left( H^{\prime
} ,\omega
^{\prime }\right) $} is a dual quasi-bialgebra {$\left( H
,\omega \right)$ }  such that $H$ is a subcoalgebra of $H^{\prime }$ and
the canonical inclusion $\sigma :H\rightarrow H^{\prime }$ is a
morphism of dual quasi-bialgebras.

{We note that by \eqref{eq:quasi-associativity} multiplication in the dual
quasi-subbialgebra $\Bbbk G(H)$ of $H$ is associative.}
\vspace{2mm}
\par Let $(H  ,\omega )$ be a dual quasi-bialgebra. It is
well-known that the category $^{H}\mathfrak{M}$ of left $H$-comodules
becomes a monoidal category as follows. Given a left $H$-comodule $V$, we
denote the left coaction of $V$ by $\rho =\rho _{V}^{l}:V\rightarrow H \otimes V,  \rho (v)=v_{-1}\otimes
v_{0}$. The tensor product of two left $H$
-comodules $V$ and $W$ is a comodule via diagonal coaction i.e. $\rho \left(
v\otimes w\right) =  v_{-1} w_{-1}  \otimes v_{0}\otimes w_{0}.$ The
unit is the trivial left $H$-comodule  $\Bbbk$,   i.e. $\rho \left( k\right) =  1_{H}\otimes k$. The associativity and
unit constraints are defined, for all $U,V,W\in {^{H}\mathfrak{M}}$ and $u\in
U,v\in V,w\in W,k\in \Bbbk ,$ by%
\begin{eqnarray}\label{form: assoc constraint}
{^{H}}a_{U,V,W}(u\otimes v\otimes w)&:=&\omega ^{-1}(u_{-1}\otimes
v_{-1}\otimes w_{-1})u_{0}\otimes (v_{0}\otimes w_{0}),  \\
l_{U}(k\otimes u):=ku\qquad &\text{and}&\qquad r_{U}(u\otimes k):=uk. \nonumber
\end{eqnarray}%
The monoidal category we have just described will be denoted by $(^{H}\mathfrak{M
}
,\otimes ,\Bbbk ,{^Ha},l,r).$

The monoidal categories $({\mathfrak{M}^{H}},\otimes ,\Bbbk ,a{^{H}},l,r)$ and
$({^{H}\mathfrak{M}^{H}},\otimes ,\Bbbk ,{^{H}}a{^{H}},l,r)$
are defined similarly. We just point out that
\begin{gather*}
a_{U,V,W}^{H}(u\otimes v\otimes w):=u_{0}\otimes (v_{0}\otimes w_{0})\omega
(u_{1}\otimes v_{1}\otimes w_{1}),\\
{^Ha^H_{U,V,W}}(u\otimes v\otimes w):=\omega ^{-1}(u_{-1}\otimes
v_{-1}\otimes w_{-1})u_{0}\otimes (v_{0}\otimes w_{0})\omega (u_{1}\otimes
v_{1}\otimes w_{1}).
\end{gather*}

\vspace{2mm}

For $\left( H, \omega \right) $   a dual
quasi-bialgebra and  $v:H^{\otimes 2}
 \rightarrow \Bbbk $  a   convolution invertible map,  define maps $m^{v}:H^{\otimes 2} \rightarrow H$ and $%
\omega ^{v}:H^{\otimes 3} \rightarrow \Bbbk $ by setting
\begin{eqnarray}
m^{v} &:&=v\ast m\ast v^{-1}  \label{form: twisted mult} \\
\omega ^{v} &:&=\left( \varepsilon \otimes {v}\right) \ast {v}\left(
H\otimes m\right) \ast \omega \ast {v}^{-1}\left( m\otimes H\right) \ast
\left( {v}^{-1}\otimes \varepsilon \right) .  \label{form: reassociator}
\end{eqnarray}%

\par  If $v$ is a gauge transformation, then the datum%
\begin{equation*}
(H, \omega)^v =(H^{{v}}, \omega^v) =\left( H,m^{v},u,\Delta ,\varepsilon ,\omega ^{v}\right)
\end{equation*}%
is a dual quasi-bialgebra with reassociator $\omega ^{v}$  called
the \emph{twisted dual quasi-bialgebra of }$H$  by $v$.

\begin{remark}\label{rem: if H cocomm}
\par (i) If $a \in \Bbbk^\times$ and $v$ is convolution invertible, then so is $av $, the composition of $v$ with multiplication
by $a$, and $av$ has inverse $a^{-1}v^{-1}$. Note that $m^{v} = m^{av}$ and $\omega^v = \omega^{av}$.
\par (ii) Note  that $(m^v)^{v^{-1}} = m$. It is straightforward to verify that
  $(\omega^v)^{v^{-1}} = \omega$, remembering that the multiplication in $(H^v, \omega^v) $ is $m^v$. Thus, since
  $v^{-1}$ is a gauge transformation for $(H^v, \omega^v)$, we have that $(H^v, \omega^v)^{v^{-1}} \cong (H,\omega)$.
\par (iii) Note that if $H$ is cocommutative then $\omega^v = \partial^2v \ast \omega$ so that $(H^v,\varepsilon^v)=(H,\partial^2v)$.
\end{remark}

\begin{lemma}\label{lm: normalize to a gauge trans} For $\left( H  ,\omega \right) $   a dual
quasi-bialgebra, suppose   $v:H\otimes
H\rightarrow \Bbbk $ is a  convolution invertible map such that $\omega^v$ as defined in (\ref{form: reassociator}) is normalized, i.e.,
satisfies
(\ref{eq:quasi-unitairitycocycle}).
  Then $a v$ is a gauge transformation for $a = v(1 \otimes 1)^{-1} \in \Bbbk^\times$.
\end{lemma}

\begin{proof}
For all $h,h^\prime \in H$,
\begin{equation*}
\varepsilon(h)\varepsilon(h^\prime) \overset{(\ref{eq:quasi-unitairitycocycle})}{=}
 \omega^v( h \otimes 1 \otimes h^\prime)
 \overset{(\ref{form: reassociator})}{=}
v(1 \otimes h^\prime)v^{-1}(h \otimes 1).
\end{equation*}
Setting $h$ and $h^\prime$ equal to $1$ in turn,   we obtain for all $h \in H$,
\begin{equation*}
\varepsilon(h) = v(1 \otimes 1) v^{-1}(h \otimes 1)= a^{-1} v^{-1}(h \otimes 1)\quad\text{  and   }\quad  \varepsilon(h ) = v(1 \otimes h ) v^{-1}(1 \otimes 1)=a v(1 \otimes h ) .
\end{equation*}
Since $a^{-1} v^{-1}(h \otimes 1)=\varepsilon(h)$ and $a v$ is the inverse of $a^{-1} v^{-1}$ then $av(h \otimes 1)=\varepsilon(h)$ also.
\end{proof}

\begin{corollary}\label{co: delta2vnormal} Let $H$ be a cocommutative bialgebra. If $v: H \otimes H \rightarrow \Bbbk$ is a
convolution invertible map such that $\partial^2v$
is a normalized cocycle, then $av$ is  normalized for $a = v(1 \otimes 1)^{-1} \in \Bbbk^\times$.
\end{corollary}

\begin{proof}
Take $\omega = \varepsilon$ in Lemma \ref{lm: normalize to a gauge trans}. Since $H$ is cocommutative, $\omega^v=\partial^2v.$
\end{proof}

\begin{proposition}
\label{pro:omegav}For $\sigma :(H,\omega_H) \rightarrow (A, \omega_A)$ a morphism of dual
quasi-bialgebras and   $v:A^{\otimes 2}\rightarrow \Bbbk $   a gauge
transformation for $A$, then {$v\circ(\sigma \otimes \sigma)$} is  a gauge transformation for $H$. Also
\begin{equation*}
\omega _{A}^{v}{\circ} \sigma^{\otimes {3}}     =\omega
_{H}^{{v}{\circ} \sigma^{\otimes 2}   }\qquad \text{and}\qquad
m_{A}^{{v}}{\circ} \sigma^{\otimes 2}   =\sigma m_{H}^{{v}{\circ}
\sigma^{\otimes 2}   }.
\end{equation*}%
Thus $\sigma :(H^{{v}{\circ} \sigma^{\otimes 2}   }, \omega_H^{{v}{\circ} \sigma^{\otimes 2}   })\rightarrow (A^{{v}}, \omega_A^v)
 $ is also a morphism of dual quasi-bialgebras between the twisted dual quasi-bialgebras obtained from $(H, \omega_H)$ and $(A, \omega_A)$.
\end{proposition}

\begin{proof}
We have%
\begin{eqnarray*}
 \omega _{A}^{v} {\circ} \sigma^{\otimes 3}
&\overset{(\ref{form: reassociator})}{=}&  \left[ \left( \varepsilon _{A}\otimes {v}\right) \ast {v}\left( A\otimes
m_{A}\right) \ast \omega _{A}\ast {v}^{-1}\left( m_{A}\otimes A\right) \ast
\left( {v}^{-1}\otimes \varepsilon _{A}\right) \right]  {\circ}  \sigma^{\otimes 3}
 \\
&=&
\left( \varepsilon _{H}\otimes {v}  \sigma^{\otimes 2}
\right) \ast {v}  \sigma^{\otimes 2}   \left( H\otimes
m_{H}\right)
\ast \omega _{H}\ast {v}^{-1}  \sigma^{\otimes 2}  \left(
m_{H}\otimes H\right) \ast \left( {v}^{-1}  \sigma^{\otimes 2}
  \otimes \varepsilon _{H}\right)%
\\
&\overset{(\ref{form: reassociator})}{=}&\omega _{H}^{{v}  \sigma^{\otimes 2}}.
\end{eqnarray*}%
Also
\begin{eqnarray*}
m_{A}^{{v}}  \sigma^{\otimes 2}
&\overset{(\ref{form: twisted mult})}{=}&
\left[ v\ast m_{A}\ast
v^{-1}\right] \left( \sigma^{\otimes 2} \right) =v  \sigma^{\otimes 2}   \ast m_{A}  \sigma^{\otimes 2}   \ast v^{-1}
\sigma^{\otimes 2}
\\
&=&v  \sigma^{\otimes 2}   \ast \sigma m_{H}\ast  [v
\sigma^{\otimes 2}  ] ^{-1}
= \sigma m_{H}^{{v}  \sigma^{\otimes 2}   }.
\end{eqnarray*}
\end{proof}

\begin{definition}
Dual quasi-bialgebras $A$ and $B$ are called \emph{quasi-isomorphic} (or
equivalent) whenever $(A, \omega_A) \cong (B^{v}, \omega_B^v)$ as dual quasi-bialgebras for some gauge
transformation ${v}\in \left( B\otimes B\right) ^{\ast }$.
\end{definition}

By Remark \ref{rem: if H cocomm}-ii., if $(A, \omega_A)  \cong (B^{v}, \omega_B^v)$,
then $(B, \omega_B) \cong (A,\omega_A)^{v^{-1}}$.

\begin{corollary}\label{coro:subgauge}
 If $\sigma :(H,\omega_H) \rightarrow (A, \omega_A)$ is a morphism of dual
quasi-bialgebras and $(A,\omega _{A})$  is  quasi-isomorphic   to an ordinary
bialgebra so  is $\left( H,\omega
_{H}\right) $.

\end{corollary}

\begin{proof}
  Suppose  $\gamma _{A}:A\otimes A\rightarrow \Bbbk $ is
a gauge transformation such that $
A^{\gamma _{A}}$ has trivial reassociator. Then $\gamma _{H}:=\gamma
_{A}\left( \sigma \otimes \sigma \right) :H\otimes H\rightarrow \Bbbk $ is a
gauge transformation,  and, by Proposition \ref{pro:omegav}, the map $\sigma :(H^{\gamma
_{H}},\omega _{H}^{\gamma _{H}})\rightarrow (A^{\gamma _{A}},\omega
_{A}^{\gamma _{A}}=\varepsilon _{A^{\otimes 3} })$ is a morphism of
dual quasi-bialgebras. Hence
\begin{equation*}
\omega _{H^{\gamma _{H}}}=\omega _{H}^{\gamma _{H}}=\omega _{A}^{\gamma
_{A}} \circ  \sigma^{ \otimes 3}  =\varepsilon
_{A^{\otimes 3} }  \circ  \sigma^{ \otimes 3}
 =\varepsilon _{H^{\otimes 3}}
\end{equation*}%
so that $H^{\gamma _{H}}$ has trivial reassociator.
\end{proof}

\section{Quasi-Yetter-Drinfeld data for dual quasi-bialgebras \label{sec:quasiYD}}

\subsection{Yetter-Drinfeld modules}\label{subsec: YD modules}

In this subsection, we first recall some facts
from \cite{Ardi-Pava2} about the
category of Yetter-Drinfeld modules for a dual quasi-bialgebra.

\begin{definition}[{\cite[Definition 3.1]{Ardi-Pava2}}]
\label{def: YD}   For $(H, \omega)$   a dual
quasi-bialgebra, the category ${_{H}^{H}\mathcal{YD}}$ of Yetter-Drinfeld
modules over $H$  is defined as follows. An object
is a tern $\left( V,\rho _{V},\vartriangleright \right) ,$ where

 $(V,\rho )$ is an object in ${^{H}\mathfrak{M}}$
and
$ \mu:H\otimes V\rightarrow V$ is a $\Bbbk$-linear map written $h \otimes v \mapsto h\vartriangleright v$
such that, for all $h,l\in H$ and $v\in V$
\begin{eqnarray}
&&\left( hl\right) \vartriangleright v=\left[
\begin{array}{c}
\omega ^{-1}\left( h_{1}\otimes l_{1}\otimes v_{-1}\right) \omega \left(
h_{2}\otimes \left( l_{2}\vartriangleright v_{0}\right) _{-1}\otimes
l_{3}\right) \\
\omega ^{-1}\left( (h_{3}\vartriangleright (l_{2}\vartriangleright
v_{0})_{0}\right) _{-1}\otimes h_{4}\otimes l_{4})\left(
h_{3}\vartriangleright \left( l_{2}\vartriangleright v_{0}\right)
_{0}\right) _{0}%
\end{array}%
\right] ,  \label{ass YD}  \\
&& 1_{H}\vartriangleright v=v  \label{unitYD} \text{   and   }\\
%
%
&& \left( h_{1}\vartriangleright v\right) _{-1}h_{2}\otimes \left(
h_{1}\vartriangleright v\right) _{0}=h_{1}v_{-1}\otimes \left(
h_{2}\vartriangleright v_{0}\right).  \label{Comp YD}
\end{eqnarray}

A morphism $f:(V,\rho ,\vartriangleright )\rightarrow (V^{\prime },\rho
^{\prime },\vartriangleright ^{\prime })$ in $_{H}^{H}\mathcal{YD}$ is a
morphism $f:(V,\rho )\rightarrow (V^{\prime },\rho ^{\prime })$ in $^{H}%
\mathfrak{M}$ such that%
\begin{equation*}
f(h\vartriangleright v)=h\vartriangleright ^{\prime }f(v).
\end{equation*}

\end{definition}

\vspace{2mm}

\begin{remark}\label{rem: weak right centre}
The category ${_{H}^{H}\mathcal{YD}}$ is isomorphic to the weak right center
of ${^{H}\mathfrak{M,}}$ see \cite[Theorem A.2.]{Ardi-Pava2}. As a
consequence ${_{H}^{H}\mathcal{YD}}$ has a prebraided monoidal structure
given as follows. The unit is $\Bbbk $ regarded as an object in ${_{H}^{H}%
\mathcal{YD}}$ via the trivial structures $\rho _{\Bbbk }\left( k\right)
=1_{H}\otimes k$ and $h\vartriangleright k=\varepsilon _{H}\left( h\right)
k. $ The tensor product is defined by
\begin{equation*}
\left( V,\rho _{V},\vartriangleright \right) \otimes \left( W,\rho
_{W},\vartriangleright \right) =\left( V\otimes W,\rho _{V\otimes
W},\vartriangleright \right)
\end{equation*}%
where $\rho _{V\otimes W}\left( v\otimes w\right) =v_{-1}w_{-1}\otimes
v_{0}\otimes w_{0}$ and%
\begin{equation}
h\vartriangleright \left( v\otimes w\right) =\left[
\begin{array}{c}
\omega \left( h_{1}\otimes v_{-1}\otimes w_{-2}\right) \omega ^{-1}\left(
\left( h_{2}\vartriangleright v_{0}\right) _{-2}\otimes h_{3}\otimes
w_{-1}\right) \\
\omega \left( \left( h_{2}\vartriangleright v_{0}\right) _{-1}\otimes \left(
h_{4}\vartriangleright w_{0}\right) _{-1}\otimes h_{5}\right) \left(
h_{2}\vartriangleright v_{0}\right) _{0}\otimes \left(
h_{4}\vartriangleright w_{0}\right) _{0}%
\end{array}%
\right] .  \label{form:YDtens}
\end{equation}%
The constraints are the same as ${^{H}\mathfrak{M}}$ i.e.
\begin{eqnarray*}
{^{H}}a_{U,V,W}(u\otimes v\otimes w) &:&=\omega ^{-1}(u_{-1}\otimes
v_{-1}\otimes w_{-1})u_{0}\otimes (v_{0}\otimes w_{0}), \\
l_{U}(k\otimes u) &:&=ku\qquad \text{and}\qquad r_{U}(u\otimes k):=uk.
\end{eqnarray*}%
viewed as morphisms in ${_{H}^{H}\mathcal{YD}}$. The prebraiding $%
c_{V,W}:V\otimes W\rightarrow W\otimes V$ is given by
\begin{equation}
c_{V,W}\left( v\otimes w\right) =\left( v_{-1}\vartriangleright w\right)
\otimes v_{0}.  \label{braiding YD}
\end{equation}
\end{remark}

\vspace{2mm}

\begin{remark}
\label{rem:tensor}The coproduct of a family $\left( V^{i}\right) _{i\in I}$
of objects in ${_{H}^{H}\mathcal{YD}}$ is the vector space $\oplus _{i\in
I}V^{i}$ regarded as an object in ${_{H}^{H}\mathcal{YD}}$ via the action
and the coaction defined by
\begin{equation*}
h\vartriangleright \left( v^{i}\right) _{i\in I}=\left( h\vartriangleright
v^{i}\right) _{i\in I}\qquad \text{and}\qquad \rho \left( \left(
v^{i}\right) _{i\in I}\right) =\sum_{i\in I}v_{-1}^{i}\otimes u_{i}\left(
v_{0}^{i}\right)
\end{equation*}%
respectively, where $u_{i}:V^{i}\rightarrow \oplus _{i\in I}V^{i}$ is the
canonical injection. Let $W\in {_{H}^{H}\mathcal{YD}}$. By the universal
property of the coproduct, the canonical morphisms $W\otimes u_{i}:W\otimes
V^{i}\rightarrow W\otimes \left( \oplus _{i\in I}V^{i}\right) $ yield a
morphism in ${_{H}^{H}\mathcal{YD}}$
\begin{equation*}
\oplus _{i\in I}\left( W\otimes V^{i}\right) \rightarrow W\otimes \left(
\oplus _{i\in I}V^{i}\right) .
\end{equation*}%
This is bijective because it is bijective at the level of vector spaces.  This proves  that the functor $W\otimes \left( -\right) :{_{H}^{H}\mathcal{%
YD}}\rightarrow {_{H}^{H}\mathcal{YD}}$ commutes with coproducts. Similarly $%
\left( -\right) \otimes W:{_{H}^{H}\mathcal{YD}}\rightarrow {_{H}^{H}%
\mathcal{YD}}$ commutes with coproducts.

Therefore we can apply \cite[Theorem 2, page 172]{Mac Lane} to construct a
left adjoint $T:{_{H}^{H}\mathcal{YD}}\rightarrow \mathrm{Mon}\left( {%
_{H}^{H}\mathcal{YD}}\right) $ of the forgetful functor. For every $V\in {%
_{H}^{H}\mathcal{YD}}$, the algebra $T\left( V\right) $ will be called the
\emph{tensor algebra of }$V$ in ${_{H}^{H}\mathcal{YD}}$. By standard
arguments one can endow $T:=T\left( V\right) $ with a bialgebra structure in
${_{H}^{H}\mathcal{YD}}$ where the comultiplication $\Delta _{T}$ and the
counit $\varepsilon _{T}$ are uniquely defined by setting $\Delta _{T}\left(
v\right) =v\otimes 1_{T}+1_{T}\otimes v$ and $\varepsilon _{T}\left(
v\right) =0.$
\end{remark}

\vspace{2mm}

The next theorem from \cite{Ardi-Pava2} gives the structure of a bosonization in this setting.

\begin{theorem}[{\cite[Theorem 5.2]{Ardi-Pava2}}]\label{teo:RsmashH} Let $\left(H,\omega _{H}\right) $ be a dual
quasi-bialgebra.

Let $(R,\mu _{R},\rho _{R},\Delta _{R},\varepsilon _{R},m_{R},u_{R})$ be a
bialgebra in $_{H}^{H}\mathcal{YD}$ and use the following notations%
\begin{eqnarray*}
h\vartriangleright r &:&=\mu _{R}\left( h\otimes r\right) ,\qquad
r_{-1}\otimes r_{0}:=\rho _{R}\left( r\right) , \\
r\cdot _{R}s &:&=m_{R}\left( r\otimes s\right) ,\qquad 1_{R}:=u_{R}\left(
1_{\Bbbk }\right) , \\
r^{1}\otimes r^{2} &:&=\Delta _{R}\left( r\right) .
\end{eqnarray*}

Consider on $B:=R\otimes H$ the following structures:%
\begin{eqnarray*}
m_{B}[(r\otimes h)\otimes (s\otimes k)] &=&\left[
\begin{array}{c}
\omega _{H}^{-1}(r_{-2}\otimes h_{1}\otimes s_{-2}k_{1})\omega
_{H}(h_{2}\otimes s_{-1}\otimes k_{2}) \\
\omega _{H}^{-1}[(h_{3}\vartriangleright s_{0})_{-2}\otimes h_{4}\otimes
k_{3}]\omega _{H}(r_{-1}\otimes (h_{3}\vartriangleright s_{0})_{-1}\otimes
h_{5}k_{4}) \\
r_{0}\cdot _{R}(h_{3}\vartriangleright s_{0})_{0}\otimes h_{6}k_{5}%
\end{array}%
\right] \\
u_{B}(k) &=&k1_{R}\otimes 1_{H} \\
\Delta _{B}(r\otimes h) &=&\omega _{H}^{-1}(r_{-1}^{1}\otimes
r_{-2}^{2}\otimes h_{1})r_{0}^{1}\otimes r_{-1}^{2}h_{2}\otimes
r_{0}^{2}\otimes h_{3} \\
\varepsilon _{B}(r\otimes h) &=&\varepsilon _{R}(r)\varepsilon _{H}(h) \\
\omega _{B}((r\otimes h)\otimes (s\otimes k)\otimes (t\otimes l))
&=&\varepsilon _{R}(r)\varepsilon _{R}(s)\varepsilon _{R}(t)\omega
_{H}(h\otimes k\otimes l).
\end{eqnarray*}%
Then $(B,\Delta _{B},\varepsilon _{B},m_{B},u_{B},\omega _{B})$ is a dual
quasi-bialgebra.
\end{theorem}

\begin{definition}[{\cite[Definition 5.4]{Ardi-Pava2}}]
For $H,R,B$ as in
Theorem \ref{teo:RsmashH}, the dual quasi-bialgebra $B$ will be called the \emph{%
bosonization of }$R$\emph{\ by }$H$ and denoted by $R\#H$. Elements of $B$ may be written $r \#h$ instead of $r \otimes h$ to emphasize that we are working in the bosonization.
\end{definition}

 \begin{remark}\label{rem:quasi}
  Let $A:= R \#H, B:= S \#L$ where $H,L$ are cosemisimple dual quasi-bialgebras and $R,S$ are bialgebras
 in the categories of Yetter-Drinfeld modules over $H$ and $L$ respectively such that $A_0 = \Bbbk \#H$ and
 $B_0 = \Bbbk \# L$. Then if $A$ and $B$ are quasi-isomorphic, so are $H$ and $L$.

 \par  For suppose  that
 there is an isomorphism $\varphi: A \rightarrow B^v$ of dual quasi-bialgebras. Since $\varphi$ is a coalgebra isomorphism,
 $\varphi(A_0) = (B^v)_0 = B_0$.

 Write  $\varphi(1\otimes h)$ as $1\otimes\varphi'(h)$ for some $\varphi'(h)\in L$. In this way we get the following commutative diagram.
\begin{equation*}
\xymatrixrowsep{.5cm}\xymatrix{H\ar[r]^-{\varphi'}\ar[d]_{\sigma_H}&L^{v(\sigma_L\otimes\sigma_L)}\ar[d]^{\sigma_L}\\
 R\#H\ar[r]^-{\varphi}&(S \#L)^v}
\end{equation*} By the same argument using $\varphi^{-1}$ we get an inverse for $\varphi'$. By Proposition \ref{pro:omegav}, the right-hand side vertical map is an injective morphism of dual quasi-bialgebras. Since  $\sigma_H$ and $\varphi$ are also morphisms of dual quasi-bialgebras we get that  $\varphi'$ also is. Thus $\varphi'$ is an isomorphism of dual quasi-bialgebras as required.
\end{remark}

The proof of the next lemma is straightforward and so is left to the reader.

\begin{lemma}\label{lm: pi}
Take the hypothesis and notations of Theorem \ref{teo:RsmashH}.
Let $\pi :R\#H\rightarrow H$ be defined by $\pi \left( r \# h\right)
=\varepsilon _{R}\left( r\right) h.$ Then $\pi$ is a morphism of dual quasi-bialgebras and
\begin{equation}
\pi \left( \left( r\# h\right) _{1}\right) \otimes \left( r\#
h\right) _{2}\# \pi \left( \left( r\# h\right) _{3}\right)
=r_{-1}h_{1}\otimes \left( r_{0}\# h_{2}\right) \otimes h_{3}.
\label{form:piIdpi}
\end{equation}
\end{lemma}

\vspace{1mm}
\begin{remark} \label{rem: with a projection} Note that for $R \# H$ as above,
the map $\sigma: H \hookrightarrow R \# H$ defined by $\sigma(h) = 1_R \# h$
is also a morphism of dual quasi-bialgebras and, for $\pi$ as defined in Lemma \ref{lm: pi},
$\pi \sigma = Id_H$.   Corollary \ref{coro:subgauge}  then implies that $H$
 is quasi-isomorphic
 to an ordinary bialgebra if and only if
$R \# H$  is.
\end{remark}

\vspace{2mm}

\subsection{Quasi-Yetter-Drinfeld data}\label{sec: qyd data}
Let $(H,\omega)$ be a dual quasi-bialgebra.  In this subsection we study one-dimensional vector spaces in $_H^H\mathcal{YD}$ and
the bialgebras in $_H^H\mathcal{YD}$ generated by these.

\begin{proposition} \label{lem: Kx}  Let $(H, \omega)$ be a dual quasi-bialgebra and let $V$ be a one-dimensional vector space.  Then $V$ is an object
in ${_{H}^{H}\mathcal{YD}}$ if and only if for all $v \in V$, $h,l \in H$,
\begin{itemize}
\item[(i)] $ V \in   {^{H}\mathfrak{M}}$   and  $\rho(v) = g \otimes v $  for some  $ g \in G(H)$;
\item[(ii)]  \label{eq: YD datum 2}   There is a unitary map $ \chi \in H^\ast $ such that  for $g$ the grouplike in $(i)$,
\begin{equation}
\chi \left( hl\right) =\omega ^{-1}\left( h_{1}\otimes l_{1}\otimes g\right)
\chi \left( l_{2}\right) \omega \left( h_{2}\otimes g\otimes l_{3}\right)
\chi \left( h_{3}\right) \omega ^{-1}\left( g\otimes h_{4}\otimes
l_{4}\right) , \label{form:chiProd}
\end{equation}
\item[(iii)] For $g$ the grouplike from part $(i)$, \begin{equation}
g \chi(h_1)h_2 = h_1 \chi(h_2) g. \label{form: YDdatum}
\end{equation}
\end{itemize}
\end{proposition}

\begin{proof}
First let $V \in {_{H}^{H}\mathcal{YD}}$.  Since $V$ is one-dimensional, there is a grouplike element $g \in H$ such that
$\rho(v) = g \otimes v$ for all $v \in V$, and also there is a map $\chi: H \rightarrow \Bbbk$ such that
 $h\rhd v = \chi(h)v$. Equation (\ref{ass YD}) of the definition of Yetter-Drinfeld modules now translates to
  (\ref{form:chiProd}) and equation (\ref{unitYD}) implies that $\chi$ is unitary.  Finally here equation (\ref{Comp YD}) of
  Definition \ref{def: YD} is equivalent to (\ref{form: YDdatum}) above.

\par Now suppose that $V$ is a one-dimensional vector space satisfying $(i)$ to $(iii)$ above. Then it is easy to see that $V$
 with coaction given by $\rho(v) = g \otimes v$ and action given by $h \rhd v
= \chi(h) v$ for all $v \in V$ is an object in $  {_{H}^{H}\mathcal{YD}}$.\end{proof}

\begin{definitions}
Let $(H, \omega)$ be a dual quasi-bialgebra.
For $g \in G(H)$ and $\chi \in H^\ast$, $\chi$ unitary,  the triple
  $((H,\omega),g,\chi)$ is called a
 \emph{%
quasi-Yetter-Drinfeld datum}, abbreviated to quasi-$YD$ datum,
whenever equations (\ref{form:chiProd}) and (\ref{form: YDdatum}) above hold.
If $q:=\chi \left( g\right) ,$  we also say that $\left(
(H,\omega),g,\chi \right) $ is a \emph{quasi-Yetter-Drinfeld datum for }$q.$
\end{definitions}

\begin{remark}\label{rem: YD datum}
(i) When $H$ is a Hopf algebra, $\omega $ is trivial and $q\neq 1$, then the previous definition reduces to \cite[Definition 2.1]{CDMM-QuantumLines}.

\par(ii) Note that Proposition \ref{lem: Kx}, roughly speaking,  says that a one-dimensional vector space $V$
 with action and coaction
 defined by $\chi$ and $g$, is an object in
${_{H}^{H}\mathcal{YD}}$ if and only if $((H,\omega),g,\chi)$ is a quasi-$YD$ datum.

\par(iii) Equation (\ref{form: YDdatum}) implies that if $((H,\omega),g,\chi)$ is a quasi-$YD$ datum
and $\ell \in G(H)$ with $\chi_H(\ell) \neq 0$, then $g \ell = \ell g$.
\end{remark}

\begin{lemma}
\label{lem:group}Let $G$ be a group and let $\omega $ be a normalized $3$%
-cocycle on $G.$ Let $g\in G$ and $\chi :\Bbbk G\rightarrow \Bbbk .$ The
following are equivalent.

\begin{itemize}
\item[$\left( i\right) $] $\left( \left( \Bbbk G,\omega \right) ,g,\chi
\right) $ is a quasi-$YD$ datum.

\item[$\left( ii\right) $] $g\in Z\left( G\right) $, $\chi $ is unitary and
(\ref{form:chiProd}) holds for all $h,\ell \in G$.
\end{itemize}
\end{lemma}

\begin{proof}
It suffices to prove that  $\left( i\right)  $ implies that $g \in Z(G)$. Let $%
h\in G.$ Then
\begin{equation*}
1=\chi \left( 1\right) =\chi \left( h^{-1}h\right) \overset{(\ref%
{form:chiProd})}{=}\omega ^{-1}\left( h^{-1}\otimes h\otimes g\right) \chi
\left( h\right) \omega \left( h^{-1}\otimes g\otimes h\right) \chi \left(
h^{-1}\right) \omega ^{-1}\left( g\otimes h^{-1}\otimes h\right),
\end{equation*}%
so that $\chi \left( h\right) $ is invertible, and  $gh = hg$ by Remark \ref{rem: YD datum}.
\end{proof}

\vspace{2mm}

\begin{definition}\label{def: morphism YD data}
For $\left( \left( H,\omega \right) ,g,\chi \right) $ and $\left( \left(
L,\alpha \right) ,l,\xi \right) $   quasi-$YD$ data,
a dual quasi-bialgebra homomorphism $
\varphi :\left( H,\omega \right) \rightarrow \left( L,\alpha \right) $ such
that $\varphi \left( g\right) =l$ and $\xi   \varphi =\chi  $  is called a
\textit{morphism of quasi-$YD$ data}.
\end{definition}

\begin{lemma}
\label{lem:qYDquotient}Let $\pi
:(A,\omega_A)\rightarrow (H,\omega_H)$ be a morphism of dual quasi-bialgebras and $\left(
(H,\omega),g,\chi \right) $   a quasi-$YD$ datum. If there exists $%
a\in G\left( A\right) $ such that $\pi \left( a\right) =g$ and $a\chi \pi
\left( b_{1}\right) b_{2}=b_{1}\chi \pi \left( b_{2}\right) a,$ for every $%
b\in A$, then $\left( (A,\omega_A),a,\chi_A:=\chi   \pi \right) $ is also a
quasi-$YD$ datum and $\pi$ is a morphism of quasi-$YD$ data.
\end{lemma}

\begin{proof}
We need only verify (\ref{form:chiProd}) for $((A, \omega),a, \chi_A)$.  For $h,l\in A$, since (\ref{form:chiProd})
holds for $\chi$, we have:
\begin{eqnarray*}
&&\omega _{A}^{-1}\left( h_{1}\otimes l_{1}\otimes a\right) \chi _{A}\left(
l_{2}\right) \omega _{A}\left( h_{2}\otimes a\otimes l_{3}\right) \chi
_{A}\left( h_{3}\right) \omega _{A}^{-1}\left( a\otimes h_{4}\otimes
l_{4}\right)
 \\
&=&\omega _{H}^{-1}\left( \pi \left( h\right) _{1}\otimes \pi \left(
l\right) _{1}\otimes g\right) \chi \left[ \pi \left( l\right) _{2}\right]
\omega _{H}\left( \pi \left( h\right) _{2}\otimes g\otimes \pi \left(
l\right) _{3}\right) \chi \left[ \pi \left( h\right) _{3}\right] \omega
_{H}^{-1}\left( g\otimes \pi \left( h\right) _{4}\otimes \pi \left( l\right)
_{4}\right)
\\
&\overset{(\ref{form:chiProd}) } {=}& \chi \left( \pi \left( h\right) \pi
\left( l\right) \right)
 = \chi \pi \left( hl\right) =\chi _{A}\left( hl\right) .
\end{eqnarray*}
\end{proof}

\begin{lemma}
\label{lem:ChiCyclic} Suppose $\left(
H,\omega \right) $ is a dual quasi-bialgebra
and  $\left(( H,\omega) ,g,\chi \right) $ is a
quasi-$YD$ datum.
Then  for   $c\in G(H)$   and $1 \leq t  $,
\begin{eqnarray}
&&\chi \left( c^{t}\right) =\chi \left( c\right) ^{t}\prod\limits_{0\leq i\leq
t-1}\left[ \omega ^{-1}\left( c^{i}\otimes c\otimes g\right) \omega \left(
c^{i}\otimes g\otimes c\right) \omega ^{-1}\left( g\otimes c^{i}\otimes
c\right) \right] ,  \label{form:datum-1}
\\
&&  \text{and, in particular,} \nonumber
  \\
&&\chi \left( g^{t}\right) =\chi \left( g\right) ^{t}\prod\limits_{0\leq i\leq
t-1}\omega ^{-1}\left( g\otimes g^{i}\otimes g\right).  \label{form:datum0}
\end{eqnarray}
\end{lemma}

\begin{proof}
 Let $s>1$ and then
by (\ref{form:chiProd})  we have%
\begin{equation}
\chi \left( c^{s-1}c\right) =\omega ^{-1}\left( c^{s-1}\otimes c\otimes g\right)
\chi \left( c\right) \omega \left(c^{s-1}\otimes g\otimes c\right) \chi
\left( c^{s-1}\right) \omega ^{-1}\left( g\otimes c^{s-1}\otimes c\right). \label{form: induction step}
\end{equation}%
Equation (\ref{form:datum-1}) now follows  by induction on $t\geq 1$. For $%
t=1$, there is nothing to prove. Let $t>1$ and assume that the statement
holds for $t-1.$ Then by (\ref{form: induction step}),%
\begin{equation*}
\chi \left( c^{t}\right) =  \chi \left( c^{t-1}\right) \left[ \chi \left( c\right)\omega ^{-1}\left(
c^{t-1}\otimes c\otimes g\right) \omega \left( c^{t-1}\otimes g\otimes
c\right) \omega ^{-1}\left( g\otimes c^{t-1}\otimes c\right) \right]
\end{equation*}
and if we then expand $\chi(c^{t-1})$ using (\ref{form:datum-1}), the result is immediate.
\end{proof}

\begin{remark} \label{rem: datum short}
If $\omega = \omega  ( H \otimes \tau)$,
then equation (\ref{form:datum-1})  simplifies to:
\begin{equation}
\chi \left( c^{t}\right) =\chi \left( c\right) ^{t}\prod\limits_{0\leq i\leq
t-1}  \omega ^{-1}\left( g\otimes c^{i}\otimes
c\right).  \label{form:datum-1 short}
\end{equation}

\section{Quantum lines}

Our first lemma will be useful in the computations to follow.
\begin{lemma}
\label{lem:omegatrick} Let $( H, \omega) $ be a dual quasi bialgebra and let $g\in G(H) $.  For all $ 0 \leq a,b,c, $
\begin{equation}
     \omega ^{-1}\left( g^{a}\otimes g^{b}\otimes g^{c}\right)
 =\prod\limits_{ 0\leq j\leq a-1}\omega ^{-1}\left( g\otimes g^{j+b }\otimes
g^{c}\right)  \omega ^{-1}\left( g\otimes
g^{j }\otimes g^{b}\right)  \omega \left(
g\otimes g^{j }\otimes g^{b+c}\right).   \label{form:omegainduct}
\end{equation}

\end{lemma}
\begin{proof}
  Let $\Phi(j):=  \omega ^{-1}\left( g\otimes g^{j+b }\otimes
g^{c}\right)  \omega ^{-1}\left( g\otimes
g^{j }\otimes g^{b}\right)  \omega \left(
g\otimes g^{j }\otimes g^{b+c}\right)  $.   The proof is by induction on $a\geq 1.$
For $a=0,1$ there is nothing to
prove. Let $a>1$ and assume the formula holds for $a-1.$
By \eqref{eq:3-cocycle} evaluated
on $g\otimes g^{a-1}\otimes g^{b}\otimes g^{c}$, we have%
\begin{equation*}
\omega \left( g^{a-1}\otimes g^{b}\otimes g^{c}\right) \omega \left(
g\otimes g^{a+b-1}\otimes g^{c}\right) \omega \left( g\otimes g^{a-1}\otimes
g^{b}\right) =\omega \left( g\otimes g^{a-1}\otimes g^{b+c}\right) \omega
\left( g^{a}\otimes g^{b}\otimes g^{c}\right)
\end{equation*}%
so that
\begin{displaymath}
\omega \left( g^{a-1}\otimes g^{b}\otimes g^{c}\right)\omega^{-1}\left( g^{a}\otimes g^{b}\otimes g^{c}\right)
 = \Phi(a-1),
\end{displaymath}
and the statement then follows from the induction assumption.
\end{proof}

Next we introduce some useful notation.

\begin{notation} \label{not: big omega}
 Let $\left( H,\omega \right) $ be a dual quasi-bialgebra.  For $U,V,W,Z$ in  $ ^{H}_H\mathcal{YD}$, we define
 $\Omega_{U,V,W,Z}: (U \otimes V) \otimes (W \otimes Z) \rightarrow (U \otimes W) \otimes (V \otimes Z)$ by
 \begin{equation}
 \Omega_{U,V,W,Z}:= a^{-1}_{U,W,V\otimes Z} (U \otimes a_{W,V,Z})
 (U \otimes c_{V,W} \otimes Z)( U \otimes a^{-1}_{V,W,Z}) a_{U,V,W \otimes Z},
\end{equation}
 where $a:= {^Ha}$ is the associativity constraint (\ref{form: assoc constraint}) in $^H \mathfrak{M}$. If $U=V=W=Z$, we write
 $\Omega_U $ in place of $\Omega_{U,U,U,U}$.
\end{notation}

 As observed in
Remark \ref{rem:tensor}, we can consider the tensor algebra $T\left(
V\right) $ of $V$ in ${_{H}^{H}\mathcal{YD}}$ for any object $V$ in ${%
_{H}^{H}\mathcal{YD}}$. Explicitly
\begin{equation*}
T\left( V\right) :=\oplus _{n\in
\mathbb{N}
}T^{n}\left( V\right) ,
\end{equation*}
where $T^{0}\left( V\right) =\Bbbk ,T^{1}\left( V\right) =V$ and, for $n>1,$
$T^{n}\left( V\right) :=V\otimes T^{n-1}\left( V\right) .$ Thus, for
instance, $T^{2}\left( V\right) =V\otimes V$ and $T^{3}\left( V\right)
=V\otimes \left( V\otimes V\right) .$ Note that the order of the brackets is
important here. \\

Let $\left(( H, \omega),g,\chi \right) $ be a quasi-$YD$ datum for  $q$.  Let $%
V=\Bbbk v$ be a one-dimensional vector space and then $(V, \rho, \mu)$ with $\rho(v) = g \otimes v$
 and $\mu(h \otimes v) = \chi(h)v$ is an object in $_H^H\mathcal{YD}$ by Remark \ref{rem: YD datum}. Set $v^{\left[ 0\right] }:=1,v^{\left[
1\right] }:=v$ and, for $n>1,$ $v^{\left[ n\right] }:=v\otimes v^{\left[ n-1%
\right] }.$ As a vector space $T\left( V\right) $ may be identified with the
polynomial ring $\Bbbk \left[ X\right] $ via the correspondence $v^{\left[ n%
\right] }\leftrightarrow X^{n} $.  However the   multiplication  is
different.

\begin{proposition} \label{prop: structure of T(V)}
With hypothesis and notations as above,
the tensor algebra $T\left( V\right) $ has basis $\left( v^{\left[ n\right]
}\right) _{n\in \mathbb{N}}$ and has the following bialgebra structure in ${_{H}^{H}\mathcal{YD}}$.
\par (i) The left coaction of $H$ on $T(V)$ is given by
\begin{equation}
\rho \left( v^{\left[ n\right] }\right) =v_{-1}^{\left[ n\right] }\otimes
v_{0}^{\left[ n\right] }=g^{n}\otimes v^{\left[ n\right] }.
\label{form: coactionTV}
\end{equation}
The left action of $H$ on $T(V)$ is given by
\begin{equation}
h\vartriangleright v^{\left[ n\right] }=\chi _{\left[ n\right] }\left(
h\right) v^{\left[ n\right] },  \label{form: actionTV}
\end{equation}%
where $\chi _{\left[ n\right] }\in H^{\ast }$ is defined iteratively by
setting $\chi_{[0]}:= \varepsilon$, and for $n \geq 1$,
\begin{equation} \label{form:chin0}
\chi _{\left[ n\right] }:=
\omega \left( -\otimes g\otimes g^{n-1}\right) \ast \chi \ast \omega
^{-1}\left( g\otimes -\otimes g^{n-1}\right) \ast \chi _{\left[ n-1\right]
}\ast \omega \left( g\otimes g^{n-1}\otimes -\right).
\end{equation}

Furthermore, for $n\geq 1$
\begin{equation}
\chi _{\left[ n\right] }=\left[ \prod\limits_{0\leq i\leq n-1}^{\ast }\omega
\left( -\otimes g\otimes g^{n-1-i}\right) \ast \chi \ast \omega ^{-1}\left(
g\otimes -\otimes g^{n-1-i}\right) \right] \ast \left[ \prod\limits_{0\leq
i\leq n-1}^{\ast }\omega \left( g\otimes g^{i}\otimes -\right) \right]
\label{form:chin}
\end{equation}
and in particular%
\begin{equation}
\chi _{\left[ n\right] }\left( g\right) =\left[ \prod\limits_{0\leq i\leq
n-1}\omega \left( g\otimes g^{i}\otimes g\right) \right] \chi \left(
g\right) ^{n} = q^n   \prod\limits_{0\leq i\leq
n-1}\omega \left( g\otimes g^{i}\otimes g\right).  \label{form: ching}
\end{equation}

\par (ii)  The algebra structure on $T(V)$ is given by $1_{T(V)} = 1_{\Bbbk} \in T^0(V)$ and

\begin{equation}
v^{\left[ a\right] }v^{\left[ b\right] }=\left[ \prod\limits_{0\leq i\leq
a-1}\omega ^{-1}\left( g\otimes g^{i}\otimes g^{b}\right) \right] v^{\left[
a+b\right] },\text{ for }  {a \geq 1}, b\in
\mathbb{N}
,  \label{form: prodTV}
\end{equation}%

\par(iii) The coalgebra structure is given by
$
\varepsilon _{T}\left( v^{\left[ n\right] }\right) =\delta _{n,0},
 $
and
\begin{equation}
 \Delta _{T}\left( v^{\left[ n\right] }\right) =\sum_{0\leq i\leq n}
\beta(i,n)
v^{\left[ i\right] }\otimes v^{\left[ n-i\right] }, \label{form: coprodTV}
\end{equation}

where
we define
\begin{equation}
\beta(i,n) =\binom{n}{i}_{q}\prod\limits_{0\leq j\leq
i-1}\omega \left( g\otimes g^{j}\otimes g^{n-i}\right) \text{ for all }0\leq
i\leq n. \label{form:n/i}
\end{equation}%
Note that $\beta(0,n) = 1= \beta(n,n)$ for all $n \geq 0$.

\end{proposition}

\begin{proof}   Equation (\ref{form: coactionTV}) follows from  the definition of the comodule
structure on the tensor product in ${_{H}^{H}\mathcal{YD}}$ in Remark \ref{rem: weak right centre}
and the fact that  $\rho \left( v\right) =g\otimes v$.

Next we compute $h\vartriangleright v^{\left[ n\right] }$ for $h\in H$. If $n=0$, then $\chi_{[0]} = \varepsilon$
satisfies (\ref{form: actionTV}). We
 prove, by induction on $n\geq 1,$ that (\ref{form: actionTV}) holds for
  $\chi _{\left[ n\right]
}\in H^{\ast }$  defined inductively   by equation (\ref{form:chin0}).  Equation (\ref{form:chin0}) gives  $\chi _{%
\left[ 1\right] }=\chi $, which satisfies (\ref{form: actionTV}). Let $n>1$ and assume that the statement holds for $%
n-1$. Then using the induction assumption and (\ref{form: coactionTV}), we have:
\begin{eqnarray*}
h\vartriangleright v^{\left[ n\right] } &=&h\vartriangleright \left(
v\otimes v^{\left[ n-1\right] }\right)
\\
&
\overset{(\ref{form:YDtens})}{=}&\left[
\begin{array}{c}
\omega \left( h_{1}\otimes v_{-1}\otimes v_{-2}^{\left[ n-1\right] }\right)
\omega ^{-1}\left( \left( h_{2}\vartriangleright v_{0}\right) _{-2}\otimes
h_{3}\otimes v_{-1}^{\left[ n-1\right] }\right)
\\
\omega \left( \left( h_{2}\vartriangleright v_{0}\right) _{-1}\otimes \left(
h_{4}\vartriangleright v_{0}^{\left[ n-1\right] }\right) _{-1}\otimes
h_{5}\right) \left( h_{2}\vartriangleright v_{0}\right) _{0}\otimes \left(
h_{4}\vartriangleright v_{0}^{\left[ n-1\right] }\right) _{0}%
\end{array}%
\right]
\\
&=&\left[
\begin{array}{c}
\omega \left( h_{1}\otimes g\otimes g^{n-1}\right) \omega ^{-1}\left(
g\otimes h_{3}\otimes g^{n-1}\right)
\\
\omega \left( g\otimes \left(
h_{4}\vartriangleright v^{\left[ n-1\right] }\right) _{-1}\otimes
h_{5}\right) \chi(h_2)v \otimes \left(
h_{4}\vartriangleright v^{\left[ n-1\right] }\right) _{0}%
\end{array}%
\right]
\\
&=&\left[
\begin{array}{c}
\omega \left( h_{1}\otimes g\otimes g^{n-1}\right)
\omega ^{-1}\left( g\otimes h_{3}\otimes g^{n-1}\right) \\
 \omega \left( g\otimes g^{n-1}
\otimes h_{5}\right) \chi \left( h_{2}\right)v\otimes \chi _{\left[ n-1\right] }\left( h_{4}\right)
v^{\left[ n-1\right] }%
\end{array}%
\right]
\\
&=&\left[ \omega \left( -\otimes g\otimes g^{n-1}\right) \ast \chi \ast
\omega ^{-1}\left( g\otimes -\otimes g^{n-1}\right) \ast \chi _{\left[ n-1%
\right] }\ast \omega \left( g\otimes g^{n-1}\otimes -\right) \right] \left(
h\right) v^{\left[ n\right] }.
\end{eqnarray*}%

Now, using this formula, we prove by induction on $n\geq 1$ that (\ref%
{form:chin}) holds. For $n=1,$ since $\omega$ is a normalized
cocycle and $g^0=1$, then the
right hand side of (\ref%
{form:chin}) is just $ \chi =\chi _{\left[ 1\right]
}$.

Let $n>1$ and assume the formula holds for $n-1$. Then%
\begin{eqnarray*}
\chi _{\left[ n\right] } &=&\omega \left( -\otimes g\otimes g^{n-1}\right)
\ast \chi \ast \omega ^{-1}\left( g\otimes -\otimes g^{n-1}\right) \ast \chi
_{\left[ n-1\right] }\ast \omega \left( g\otimes g^{n-1}\otimes -\right) \\
\\
&=&
\left[
\begin{array}{c}
[\omega \left( -\otimes g\otimes g^{n-1}\right) \ast \chi \ast \omega
^{-1}\left( g\otimes -\otimes g^{n-1}\right)] \ast \\
\left[ \prod\limits_{0\leq i\leq n-2}^{\ast }\omega \left( -\otimes g\otimes
g^{n-2-i}\right) \ast \chi \ast \omega ^{-1}\left( g\otimes -\otimes
g^{n-2-i}\right) \right] \ast \left[ \prod\limits_{0\leq i\leq n-2}^{\ast
}\omega \left( g\otimes g^{i}\otimes -\right) \right] \ast \\
\omega \left( g\otimes g^{n-1}\otimes -\right)%
\end{array}%
\right] \\
&=&\left[ \prod\limits_{0\leq i\leq n-1}^{\ast }\omega \left( -\otimes
g\otimes g^{n-1-i}\right) \ast \chi \ast \omega ^{-1}\left( g\otimes
-\otimes g^{n-1-i}\right) \right] \ast \left[ \prod\limits_{0\leq i\leq
n-1}^{\ast }\omega \left( g\otimes g^{i}\otimes -\right) \right] ,
\end{eqnarray*}%
and so (\ref%
{form:chin}) holds for all $n \geq 1$.
We note that if (\ref%
{form:chin})  is applied to a cocommutative element, then since the product for $0 \leq i \leq n-1$ is the same as taking the product over $0 \leq n-1-i \leq n-1$,
\begin{equation}
\chi _{\left[ n\right] }=    \chi^n \ast \left [ \prod\limits_{0\leq i\leq n-1}^{\ast }\omega
\left( -\otimes g\otimes g^{i}\right) \ast \omega ^{-1}\left(
g\otimes -\otimes g^{i}\right)  \ast \omega \left( g\otimes g^{i}\otimes -\right)\right ].
\label{form:chincocomm}
\end{equation}
Equation (\ref{form: ching}) follows immediately.

\par (ii) Let $b\in\mathbb{N}$ and we prove by induction on $a\geq 1$ that (\ref{form: prodTV}) holds. For $%
a=1$, we have by definition $v^{\left[ a\right] }v^{\left[ b\right] }=vv^{\left[ b\right]
}=v\otimes v^{\left[ b\right] }=v^{\left[ 1+b\right] }=v^{\left[ a+b\right]
} $. Let $a>1$ and assume   (\ref{form: prodTV})  for $a-1$. Then
\begin{eqnarray*}
v^{\left[ a\right] }v^{\left[ b\right] } &=&\left( v\otimes v^{\left[ a-1%
\right] }\right) v^{\left[ b\right] }=\left( vv^{\left[ a-1\right] }\right)
v^{\left[ b\right] }\overset{(\ref{form: assoc constraint})}{=} \omega ^{-1}\left( v_{-1}\otimes v_{-1}^{\left[ a-1%
\right] }\otimes v_{-1}^{\left[ b\right] }\right) v_{0}\left( v_{0}^{\left[
a-1\right] }v_{0}^{\left[ b\right] }\right)
\\
&\overset{(\ref{form: coactionTV})}{=}&\omega ^{-1}\left( g\otimes
g^{a-1}\otimes g^{b}\right) v\left( v^{\left[ a-1\right] }v^{\left[ b\right]
}\right)
\\
&=&\omega ^{-1}\left( g\otimes g^{a-1}\otimes g^{b}\right) v\left[
\prod\limits_{0\leq i\leq a-2}\omega ^{-1}\left( g\otimes g^{i}\otimes
g^{b}\right) \right] v^{\left[ a-1+b\right] }
 \\
&=&\left[ \prod\limits_{0\leq i\leq a-1}\omega ^{-1}\left( g\otimes
g^{i}\otimes g^{b}\right) \right] vv^{\left[ a-1+b\right] }=\left[
\prod\limits_{0\leq i\leq a-1}\omega ^{-1}\left( g\otimes g^{i}\otimes
g^{b}\right) \right] v^{\left[ a+b\right] },
\end{eqnarray*}%

and so we have proved that (\ref{form: prodTV}) holds for $a \geq 1$.

\par (iii) We wish to show that
\begin{equation*}
\Delta _{T}\left( v^{\left[ n\right] }\right) =\sum_{0\leq i\leq n}\binom{n}{%
i}_{q}  \prod\limits_{0\leq j\leq
i-1}\omega \left( g\otimes g^{j}\otimes g^{n-i}\right)    v^{\left[ i\right] }\otimes v^{\left[ n-i\right] },
\end{equation*}%
where if $i=0$, the empty product  is defined to be $1$.  If $n=0$ the formula holds since
$\Delta_T(v^{[0]}) = 1 \otimes 1$.  If $n=1$, the formula holds since $\Delta _{T}\left( v^{\left[ 1\right] }\right) =v\otimes
1+1\otimes v$, $\binom{1}{0}_q = \binom{1}{1}_q = 1$, and $ \omega(- \otimes 1 \otimes -) = 1$.
Let $n>1$ and suppose that the formula holds for $n-1$.
Then
\begin{eqnarray*}
&& \Delta _{T}\left( v^{\left[ n\right] }\right)
= \Delta _{T}m_{T}\left(
v\otimes v^{\left[ n-1\right] }\right)
= \left( m_{T}\otimes m_{T}\right) \Omega_{T} \left( \Delta
_{T}\otimes \Delta _{T}\right) \left( v\otimes v^{\left[ n-1\right] }\right)
\\
&=&\left( m_{T}\otimes m_{T}\right) \Omega_{T }
 \left[ \left(
v\otimes 1+1\otimes v\right) \otimes
   (\sum_{0\leq i\leq n-1}\beta(i,n-1)   v^{\left[ i\right] }\otimes v^{\left[ n-1-i\right] }) \right].
\end{eqnarray*}

From the definition of $\Omega_T$, it is easily seen that:
\begin{equation*}
 \Omega_T((v \otimes 1) \otimes ( v^{\left[ i\right] }\otimes v^{\left[ n-1-i\right]} ))\\
= \omega(g \otimes g^i \otimes g^{n-1-i}) ((v \otimes v^{[i]}) \otimes (1 \otimes   v^{\left[ n-1-i\right]} ))
\end{equation*}
and
\begin{equation*}
\Omega_T((1 \otimes v) \otimes (v^{[i]} \otimes v^{[n-1-i]}) = \chi_{[i]}(g) \omega(g\otimes g^i \otimes g^{n-1-i})
 \omega^{-1}(g^i \otimes g \otimes g^{n-1-i}) (1 \otimes v^{[i]}) \otimes ( v \otimes v^{[n-1-i]}).
\end{equation*}

Then

\begin{eqnarray*}
&&\Delta _{T}\left( v^{[n]}\right)  \\
&&=\sum_{0\leq i\leq n-1}\beta (i,n-1)\left[
\begin{array}{c}
\omega (g\otimes g^{i}\otimes g^{n-1-i})v^{[i+1]}\otimes v^{[n-1-i]} \\
+\chi _{\lbrack i]}(g)\omega (g\otimes g^{i}\otimes g^{n-1-i})\omega
^{-1}(g^{i}\otimes g\otimes g^{n-1-i})v^{[i]}\otimes v^{[n-i]}]%
\end{array}%
\right] .
\end{eqnarray*}

The coefficient of $v^{[0]} \otimes v^{[n]}$ in the expression above is $\chi_{[0]}(g)\omega(g \otimes 1 \otimes g^{n-1})
\omega^{-1}(1 \otimes g \otimes g^{n-1}) = 1 = \beta(0,n)$ and similarly the coefficient of $v^{[n]} \otimes v^{[0]}$ is $
\beta(n-1,n-1) =1$.
For $    1 \leq j \leq n-1$, we compute the coefficient of $v^{[j]} \otimes v^{[n-j]}$ in this expression to be:
\begin{eqnarray*}
&&\beta(j-1,n-1)\omega(g \otimes g^{j-1} \otimes g^{n-j}) + \beta(j,n-1)\omega(g \otimes g^j \otimes g^{n-1-j})
 \chi_{[j]}(g) \omega^{-1}(g^j \otimes g \otimes g^{n-1-j})
\\
&=& \left [ \binom{n-1}{j-1}_q \prod_{0 \leq k \leq j-2} \omega(g \otimes g^k \otimes g^{n-j}) \right ]
\omega(g \otimes g^{j-1} \otimes g^{n-j})
\\
&& + \left [ \binom{n-1}{j}_q \prod_{0 \leq i \leq j-1} \omega(g \otimes g^i \otimes g^{n-1-j}) \right ]
\omega(g \otimes g^j \otimes g^{n-1-j} )\chi_{[j]}(g) \omega^{-1}(g^j \otimes g \otimes g^{n-1-j})\\
&=&   \binom{n-1}{j-1}_q \prod_{0 \leq k \leq j-1} \omega(g \otimes g^k \otimes g^{n-j} )
   \\
&& +  \binom{n-1}{j}_q \left[ \prod_{0 \leq i \leq j} \omega(g \otimes g^i \otimes g^{n-1-j} )\right]
  \chi_{[j]}(g) \omega^{-1}(g^j \otimes g \otimes g^{n-1-j})\\
  &=&   \binom{n-1}{j-1}_q \prod_{0 \leq k \leq j-1} \omega(g \otimes g^k \otimes g^{n-j} )
   \\
&& +  \binom{n-1}{j}_q \left[ \prod_{0 \leq s \leq j-1} \omega(g \otimes g^{s+1} \otimes g^{n-1-j} )\right]
  \chi_{[j]}(g) \omega^{-1}(g^j \otimes g \otimes g^{n-1-j}).
\end{eqnarray*}

By Lemma \ref{lem:omegatrick}
 \begin{displaymath}
  \omega ^{-1}\left( g^{j}\otimes g\otimes g^{n-1-j}\right)
 =\prod\limits_{ 0\leq s\leq j-1}\omega ^{-1}\left( g\otimes g^{s+1 }\otimes
g^{n-1-j}\right)  \omega ^{-1}\left( g\otimes
g^{s }\otimes g\right)  \omega \left(
g\otimes g^{s }\otimes g^{n-j}\right)
\end{displaymath}
  so that the last summand in the  expression above becomes:
\begin{eqnarray*}
&& \binom{n-1}{j}_q
  \chi_{[j]}(g) \prod_{0 \leq s \leq j-1} \omega^{-1}(g \otimes g^s \otimes g) \prod_{0 \leq t \leq j-1}\omega(g \otimes g^t \otimes g^{n-j}) \\
  &\overset{\eqref{form: ching}}=& \binom{n-1}{j}_q q^j \prod_{0 \leq s \leq j-1} \left (\omega(g \otimes g^s \otimes g)
   \omega^{-1}(g \otimes g^s \otimes g)\right )  \prod_{0 \leq t \leq j-1}\omega(g \otimes g^t \otimes g^{n-j})\\
   &=& \binom{n-1}{j}_q q^j  \prod_{0 \leq t \leq j-1}\omega(g \otimes g^t \otimes g^{n-j}).
\end{eqnarray*}

Thus the coefficient of $v^{[j]} \otimes v^{[n-j]}$ in $\Delta _{T}\left( v^{\left[ n\right] }\right)$ is
\begin{eqnarray*}
&&  \left [ \binom{n-1}{j-1}_q
 +  \binom{n-1}{j}_q q^j \right ] \prod_{0 \leq t \leq j-1}\omega(g \otimes g^t \otimes g^{n-j})
= \binom{n}{j}_q  \prod_{0 \leq t \leq j-1}\omega(g \otimes g^t \otimes g^{n-j}),
\end{eqnarray*}
and this is indeed $\beta(j,n)$ as required. It is then straightforward to see that $\varepsilon_T(v^{[n]} )
= \delta_{n,0}$.
\end{proof}

 The next technical results allow us to construct a bialgebra quotient of the tensor algebra.

\begin{proposition}
\label{pro:BialgQuotient} Let $\left( A,m_{A},u_{A},\Delta _{A},\varepsilon
_{A}\right) $ be a bialgebra in an abelian prebraided monoidal category $\left(
\mathcal{M},\otimes ,\mathbf{1},a,l,r,c\right) $ where the tensor functors are additive and right exact. Let $\left(
I,i_{I}:I\rightarrow A\right) $ be a subobject of $A$ in $\mathcal{M}$ such
that
\begin{gather}
\left( p_{R}\otimes p_{R}\right)\circ   \Delta _{A} \circ i_{I}=0,
\label{form:coid1} \\
\varepsilon _{A} \circ i_{I}=0,  \label{form:coid2} \\
p_{R} \circ m_{A} \circ i_{K}=0.  \label{form:coid3}
\end{gather}%
where $R:=A/I$, $p_{R}:A\rightarrow R$ denotes the canonical projection and $%
\left( K,i_{K}:K\rightarrow A\otimes A\right) :=\mathrm{Ker}\left(
p_{R}\otimes p_{R}\right) $. Then there are maps $m_{R},u_{R},\Delta
_{R},\varepsilon _{R}$ such that $\left( R,m_{R},u_{R},\Delta
_{R},\varepsilon _{R}\right) $ is a bialgebra in $\left( \mathcal{M},\otimes
,\mathbf{1},a,l,r,c\right) $ and $p_{R}$ is a bialgebra morphism.
\end{proposition}

\begin{proof}
{In this proof we omit the constraints as in view of the coherence theorem
they take care of themselves.} By (\ref{form:coid1}) and (\ref{form:coid2}),
there are morphisms
\begin{equation*}
\Delta _{R}:R\rightarrow R\otimes R\qquad \text{and}\qquad \varepsilon
_{R}:R\rightarrow \Bbbk
\end{equation*}%
defined by $\Delta _{R}  p_{R}=\left( p_{R}\otimes p_{R}\right)
\Delta _{A}$ and $\varepsilon _{R} ( p_{R})=\varepsilon _{A}.$ The first
equality yields%
\begin{eqnarray*}
\left( \Delta _{R}\otimes R\right)   \Delta _{R} p_{R} &=&\left(
\left( p_{R}\otimes p_{R}\right) \otimes p_{R}\right)   \left( \Delta
_{A}\otimes A\right)   \Delta _{A} \\
&=&\left( p_{R}\otimes \left( p_{R}\otimes p_{R}\right) \right)   \left(
A\otimes \Delta _{A}\right)   \Delta _{A}=\left( R\otimes \Delta
_{R}\right)   \Delta _{R} p_{R}
\end{eqnarray*}%
so that $\left( \Delta _{R}\otimes R\right)   \Delta _{R}=\left(
R\otimes \Delta _{R}\right)   \Delta _{R}.$ The other equality leads to
counitarity of $\Delta _{R}.$ Since the tensor functors are right exact, we
have that $p_{R}\otimes p_{R}$ is an epimorphism and hence $\left( R\otimes
R,p_{R}\otimes p_{R}\right) =\mathrm{Coker}\left( i_{K}\right) .$ Thus, by (%
\ref{form:coid3}), we have that there is a unique map $m_{R}:R\otimes
R\rightarrow R$ such that $m_{R}( p_{R})=\left( p_{R}\otimes p_{R}\right)
  m_{A}.$ Set $u_{R}:= p_{R}( u_{A}).$ The first equality yields%
\begin{equation*}
m_{R}  \left( m_{R}\otimes R\right)    p_{R}^{\otimes 3} =m_{R}  \left( R\otimes m_{R}\right)
 p_{R}^{\otimes 3}
\end{equation*}%
so that, by right exactness of tensor functors, we get $m_{R}  \left(
m_{R}\otimes R\right) =m_{R}  \left( R\otimes m_{R}\right) .$ Similarly
one gets $m_{R}  \left( u_{R}\otimes R\right) =l_{R}$ and $m_{R}
\left( R\otimes u_{R}\right) =r_{R}.$ Finally, we have%
\begin{eqnarray*}
&&\left( m_{R}\otimes m_{R}\right)   \left( R\otimes c_{R,R}\otimes
R\right)   \left( \Delta _{R}\otimes \Delta _{R}\right)   \left(
p_{R}\otimes p_{R}\right)  \\
&=&\left( p_{R}\otimes p_{R}\right)  \left( m_{A}\otimes m_{A}\right)
 \left( A\otimes c_{A,A}\otimes A\right)   \left( \Delta
_{A}\otimes \Delta _{A}\right)  \\
&=&\left( p_{R}\otimes p_{R}\right)  {\Delta _{A}  m_{A}}  =
 \Delta _{R}  m_{R}  \left( p_{R}\otimes p_{R}\right).
\end{eqnarray*}%
Since $p_{R}\otimes p_{R}$ is an epimorphism, we get $\left(
m_{R}\otimes m_{R}\right)   \left( R\otimes c_{R,R}\otimes R\right)
  \left( \Delta _{R}\otimes \Delta _{R}\right) =\Delta _{R}  m_{R}.$
Thus $\left( R,m_{R},u_{R},\Delta _{R},\varepsilon _{R}\right) $ is a
bialgebra in $\left( \mathcal{M},\otimes ,\mathbf{1},a,l,r,c\right) .$
Clearly $p_{R}$ is a bialgebra morphism.
\end{proof}

\begin{lemma}
\label{lem:generated}Let $(H, \omega)$ be a dual quasi-bialgebra and let $I$ be an
ideal of a bialgebra $A$ in ${_{H}^{H}\mathcal{YD}}$. Let $z,u\in A$ and
assume $\Delta _{A}\left( u\right) \in A\otimes I+I\otimes A.$ Then $\Delta
_{A}\left( zu\right) \in A\otimes I+I\otimes A.$
\end{lemma}

\begin{proof}
Since $\Delta_A(u) \in A \otimes I + I \otimes A$, then
\begin{eqnarray*}
&& \Delta _{A}\left( zu\right)= \Delta _{A}m_{A}\left( z\otimes u\right)  =
 \left( m_{A}\otimes m_{A}\right)\Omega_{A }\left( \Delta
_{A}\otimes \Delta _{A}\right) \left( z\otimes u\right) \\
& \in &\left( m_{A}\otimes m_{A}\right)\Omega_{A }[(A\otimes A)\otimes (A \otimes I) + (A \otimes A )
\otimes (I \otimes A)]\\
& \subseteq &\left( m_{A}\otimes m_{A}\right) [(A\otimes A)\otimes (A \otimes I) + (A \otimes I )
\otimes (A \otimes A)] \subseteq A \otimes I + I \otimes A. \  \end{eqnarray*}
 \end{proof}

Let $\left( H,g,\chi \right) $ be a quasi-$YD$ datum for $q$ with
  $q$   a primitive $N$-th root of unity with $N>0$.  Let $V=\Bbbk v \in {_H^H\mathcal{YD}}$ with coaction and action defined
  by $g$ and $\chi$ as in Remark \ref{rem: YD datum}. Let $\left( v^{\left[ n\right]
}\right) _{n\in
\mathbb{N}}$ be the basis of $T:=T\left( V\right) $ considered at the beginning of this
section.

Let $I$ be the two-sided ideal of $T$ generated by $v^{\left[ N\right] }$
, i.e,  $I=:T\left( IT\right) $. Since $vv^{\left[ n\right]
}=v^{\left[ n+1\right] }$,  by (\ref{form: prodTV}),   $I$ is
the vector space with basis $\left( v^{\left[ n\right] }\right) _{n\geq N}.$
Thus, $ T/I$ identifies with $K\left[ X\right] /\left( X^{N}\right) .$ By
formulas  (\ref{form: actionTV}) and (\ref{form: coactionTV}), we deduce that $%
I$ is a subobject of $T$ in ${_{H}^{H}\mathcal{YD}}$. Hence $I$ is a
two-sided ideal of $T$ in ${_{H}^{H}\mathcal{YD}}$. Moreover $R:=T/I$ with the induced structures is in ${%
_{H}^{H}\mathcal{YD}}$  so that the
canonical projection $p_{R}:T\rightarrow R$ is in ${_{H}^{H}\mathcal{YD}}$.

To check (\ref{form:coid1}) for $I$ we must show that
\begin{equation*}
\Delta _{T}\left( v^{\left[ n\right] }\right) \in T\otimes I+I\otimes T,%
\text{ for every }n\geq N.
\end{equation*}%

For $n=N$, this follows from (\ref{form: coprodTV})  and the fact that, since $q$ has order $N$, then $\binom{N}{i}_{q}=0$ for $i\neq 0,N$.
For $n\geq N+1,$ in view of (\ref{form: prodTV}) we have%
\begin{equation*}
v^{\left[ n\right] }=\prod\limits_{0\leq i\leq n-N-1}\omega \left( g\otimes
g^{i}\otimes g^{N}\right) v^{\left[ n-N\right] }v^{\left[ N\right] }.
\end{equation*}%
Hence,  Lemma \ref{lem:generated} implies that $\Delta _{T}\left( v^{\left[ n\right] }\right) \in T\otimes
I+I\otimes T$.

 Since by
Proposition \ref{prop: structure of T(V)},  $\varepsilon_T(v^{[n]}) = \delta_{n,0}$ and $N\neq0$,
it is clear that $\varepsilon _{T}\left( I\right) =0.$
Since $I$ is a two-sided ideal of $T=T\left( V\right) $, we have that $%
m_{T}\left( T\otimes I+I\otimes T\right) \subseteq I.$ Since $\mathrm{Ker}%
\left( p_{R}\otimes p_{R}\right) =T\otimes I+I\otimes T,$ we deduce that $%
m_{T}\left( \mathrm{Ker}\left( p_{R}\otimes p_{R}\right) \right) \subseteq
I. $ By Proposition \ref{pro:BialgQuotient}, there are maps $%
m_{R},u_{R},\Delta _{R},\varepsilon _{R}$ such that $\left(
R,m_{R},u_{R},\Delta _{R},\varepsilon _{R}\right) $ is a bialgebra in $%
\left( {_{H}^{H}\mathcal{YD}},\otimes ,\mathbf{\Bbbk },{^{H}}a,l,r,c\right) $
and $p_{R}$ is a bialgebra morphism.

  Recall that the Iverson bracket $\left[[ P ]\right] $ is a notation that denotes a number that is $1$
if the condition $P$ in double square brackets is satisfied, and $0$ otherwise.

By the above we have the following result.

\begin{theorem}\label{teo:RsmachHthin}
\label{teo:R}   Let $\left( \left( H,\omega \right) ,g,\chi \right) $  be a
quasi-$YD$ datum  for $q>1$, a primitive $N^{th}$ root of unity.

\par (i) There is a bialgebra $R=R\left( \left( H,\omega \right) ,g,\chi \right)
$ in ${_{H}^{H}\mathcal{YD}}$ with basis $\left( x^{\left[ n\right] }\right)
_{0\leq n\leq N-1}$ and structure given as follows:%
\begin{eqnarray*}
\rho \left( x^{\left[ n\right] }\right) &:&=g^{n}\otimes x^{\left[ n\right]
}, \\
h\vartriangleright x^{\left[ n\right] } &:&=\chi _{\left[ n\right] }\left(
h\right) x^{\left[ n\right] }, \text{ where } \chi _{\left[ n\right] } \in H^\ast \text{ is defined in (\ref{form:chin}) }, \\
1_{R} &:&=x^{\left[ 0\right] }, \\
m_{R}\left( x^{\left[ a\right] }\otimes x^{\left[ b\right] }\right)
&=&  [[a+b \leq N-1  ]] \left[\prod\limits_{0\leq i\leq a-1}\omega ^{-1}\left(
g\otimes g^{i}\otimes g^{b}\right)\right] x^{\left[ a+b\right] }\text{ when }%
a,b\geq 0,
 \\
\Delta _{R}\left( x^{\left[ n\right] }\right) &=&
\sum_{0 \leq i \leq n} \beta(i,n) x^{[i]} \otimes x^{[n-i]},
 \text{ where } \beta(i,n)\text{  is defined in } (\ref{form:n/i}),
 \\
\varepsilon _{R}\left( x^{\left[ n\right] }\right) &:&=\delta _{n,0}.
\end{eqnarray*}
\par (ii) For $R$ the bialgebra in $_H^H\mathcal{YD}$ from (i), let $B:=R \#H$, the bosonization
of $R$ by $H$.  Then

\begin{equation*}
B_{0}\subseteq \Bbbk 1_{R}\otimes H.
\end{equation*}

\end{theorem}

\begin{proof}
(i) Take $R:=T\left( V\right) /I$ as above   and set $x^{[n]}:=v^{[n]}+I$.
\par (ii) For $0 \leq n <N$,
let $R_{\left[ n\right] }:=\oplus _{0\leq a\leq n}\Bbbk x^{\left[ a\right]
}. $ Then, by the structure maps for $R$ in (i), $R_{\left[ n\right] }$ is
a subobject of $R$ in ${_{H}^{H}\mathcal{YD}}$ such that
\begin{equation*}
\Delta _{R}\left( R_{\left[ n\right] }\right) \subseteq \sum_{0\leq i\leq n}R_{%
\left[ i\right] }\otimes R_{\left[ n-i\right] }.
\end{equation*}%
Set $B_{\left[ n\right] }:=R_{\left[ n\right] }\otimes H.$
By the structure maps for $B$ in Theorem \ref{teo:RsmashH}, we see %

\begin{equation*}
\Delta _{B}\left( B_{\left[ n\right] }\right) \subseteq \sum_{0\leq i\leq n}B_{%
\left[ i\right] }\otimes B_{\left[ n-i\right] }.
\end{equation*}%
Since $B=\cup _{n\in\mathbb{N}}B_{\left[ n\right] },$ we have proved that $B$ is a filtered coalgebra so
that, by \cite[Proposition 11.1.1, page 226]{Sw},
\begin{equation*}
B_{0}\subseteq B_{\left[ 0\right] }=R_{\left[ 0\right] }\otimes H=\Bbbk x^{%
\left[ 0\right] }\otimes H=\Bbbk 1_{R}\otimes H.
\end{equation*}%
\end{proof}

Note that the result in the previous theorem still holds formally if ${q}=1$ but is not so interesting,
 since, in this case $R$ collapses to the base field $\Bbbk$.

\begin{definition}
 Let $\left( \left( H,\omega \right) ,g,\chi \right) $ be a
quasi-$YD$ datum for $q \neq 1$, a primitive $N$-th root of unity. The bialgebra $R=R\left( \left( H,\omega \right) ,g,\chi \right)
$ of the previous theorem will be called a \emph{quantum line} for the given datum.
\end{definition}

\begin{proposition}
\label{pro:SR}The bialgebra $R$ from Theorem \ref{teo:RsmachHthin} is a Hopf
algebra in ${_{H}^{H}\mathcal{YD}}$ with bijective antipode $%
S_{R}:R\rightarrow R$ defined by
\begin{equation*}
S_{R}\left( x^{\left[ n\right] }\right) =\left( -1\right) ^{n}\chi \left(
g\right) ^{\frac{n\left( n-1\right) }{2}}x^{\left[ n\right] }\text{ for }%
0\leq n\leq N-1.
\end{equation*}
\end{proposition}

\begin{proof}
Consider the basis $\left( x^{\left[ n\right] }\right) _{0\leq n\leq N-1}$
of the bialgebra $R=R\left( \left( H,\omega \right) ,g,\chi \right) $ in ${%
_{H}^{H}\mathcal{YD}}$. We want to define a linear map $S_{R}:R\rightarrow R$
on the basis which a posteriori is expected to be antimultiplicative in ${%
_{H}^{H}\mathcal{YD}}$. Set $S_{R}\left( x^{\left[ 1\right] }\right) =-x^{%
\left[ 1\right] }.$ Then, for $1<n\leq N-1,$ we have
\begin{eqnarray*}
S_{R}\left( x^{\left[ n\right] }\right) &=&S_{R}m_{R}\left( x^{\left[ 1%
\right] }\otimes x^{\left[ n-1\right] }\right) =m_{R}\left( S_{R}\otimes
S_{R}\right) c_{R,R}\left( x^{\left[ 1\right] }\otimes x^{\left[ n-1\right]
}\right) \\
&=&m_{R}\left( S_{R}\otimes S_{R}\right) \left( x_{-1}^{\left[ 1\right]
}\vartriangleright x^{\left[ n-1\right] }\otimes x_{0}^{\left[ 1\right]
}\right) \\
&=&m_{R}\left( S_{R}\otimes S_{R}\right) \left( g\vartriangleright x^{\left[
n-1\right] }\otimes x^{\left[ 1\right] }\right) =\chi _{\left[ n-1\right]
}\left( g\right) m_{R}\left( S_{R}\otimes S_{R}\right) \left( x^{\left[ n-1%
\right] }\otimes x^{\left[ 1\right] }\right) \\
&=&\chi _{\left[ n-1\right] }\left( g\right) S_{R}\left( x^{\left[ n-1\right]
}\right) S_{R}\left( x^{\left[ 1\right] }\right) =-\chi _{\left[ n-1\right]
}\left( g\right) S_{R}\left( x^{\left[ n-1\right] }\right) x^{\left[ 1\right]
},
\end{eqnarray*}%
where $c_{R,R}:R\otimes R\rightarrow R\otimes R$ denotes the braiding of ${%
_{H}^{H}\mathcal{YD}}$ evaluated in $R$. Let us check that this forces%
\begin{equation*}
S_{R}\left( x^{\left[ n\right] }\right) =\left( -1\right) ^{n}\chi \left(
g\right) ^{\frac{n\left( n-1\right) }{2}}x^{\left[ n\right] }\text{ for }%
0\leq n\leq N-1.
\end{equation*}%
For $n=0,1$ the formula trivially holds. Let $n$ with $1<n\leq N-1$ such
that the formula holds for $n-1.$ Then%
\begin{eqnarray*}
S_{R}\left( x^{\left[ n\right] }\right) &=&-\chi _{\left[ n-1\right] }\left(
g\right) S_{R}\left( x^{\left[ n-1\right] }\right) x^{\left[ 1\right] } \\
&=&-\left[ \prod\limits_{0\leq i\leq n-2}\omega \left( g\otimes g^{i}\otimes
g\right) \right] \chi \left( g\right) ^{n-1}\left( -1\right) ^{n-1}\chi
\left( g\right) ^{\frac{\left( n-1\right) \left( n-2\right) }{2}}x^{\left[
n-1\right] }x^{\left[ 1\right] } \\
&=&\left[ \prod\limits_{0\leq i\leq n-2}\omega \left( g\otimes g^{i}\otimes
g\right) \right] \left( -1\right) ^{n}\chi \left( g\right) ^{\frac{n\left(
n-1\right) }{2}}x^{\left[ n-1\right] }x^{\left[ 1\right] } \\
&=&\left[ \prod\limits_{0\leq i\leq n-2}\omega \left( g\otimes g^{i}\otimes
g\right) \right] \left( -1\right) ^{n}\chi \left( g\right) ^{\frac{n\left(
n-1\right) }{2}}\prod\limits_{0\leq i\leq n-2}\omega ^{-1}\left( g\otimes
g^{i}\otimes g\right) x^{\left[ n\right] } \\
&=&\left( -1\right) ^{n}\chi \left( g\right) ^{\frac{n\left( n-1\right) }{2}%
}x^{\left[ n\right] }.
\end{eqnarray*}%
We have%
\begin{eqnarray*}
S_{R}\left( \left( x^{\left[ n\right] }\right) ^{1}\right) \left( x^{\left[ n%
\right] }\right) ^{2} &=&\sum_{0\leq i\leq n}\beta (i,n)S_{R}\left(
x^{[i]}\right) x^{[n-i]} \\
&=&\sum_{0\leq i\leq n}\beta (i,n)\left( -1\right) ^{i}\chi \left( g\right)
^{\frac{i\left( i-1\right) }{2}}x^{\left[ i\right] }x^{[n-i]} \\
&=&\sum_{0\leq i\leq n}\beta (i,n)\left( -1\right) ^{i}\chi \left( g\right)
^{\frac{i\left( i-1\right) }{2}}\prod\limits_{0\leq j\leq i-1}\omega
^{-1}\left( g\otimes g^{j}\otimes g^{n-i}\right) x^{\left[ n\right] }.
\end{eqnarray*}%
{But
\begin{equation}\label{eqn: beta}
 \beta(i,n)
  \prod\limits_{0\leq j\leq i-1}\omega ^{-1}\left( g\otimes g^{j}\otimes g^{n-i}\right)
= \binom{n}{i}_{q}
\end{equation}
so that
\begin{displaymath}
S_{R}\left( \left( x^{\left[ n\right] }\right) ^{1}\right) \left( x^{\left[ n%
\right] }\right) ^{2} = \left[ \sum_{0\leq i\leq n}\binom{n}{i}_{q}\left( -1\right) ^{i}q^{\frac{%
i\left( i-1\right) }{2}}\right] x^{\left[ n\right] }.
\end{displaymath}
}

By \cite[Proposition IV.2.7]{Kassel-Quantum} we have that%
\begin{equation*}
\sum_{0\leq i\leq n}\binom{n}{i}_{q}\left( -1\right) ^{i}q^{\frac{i\left(
i-1\right) }{2}}a^{n-i}X^{i}=\prod\limits_{0\leq i\leq n-1}\left(
a-q^{i}X\right)
\end{equation*}%
for any scalar $a$ and variable $X.$ If we take $a=1$ and evaluate this
polynomial in $X=1,$ we get $\sum_{0\leq i\leq n}\binom{n}{i}_{q}\left(
-1\right) ^{i}q^{\frac{i\left( i-1\right) }{2}}=\delta _{n,0}.$ Hence%
\begin{equation*}
S_{R}\left( \left( x^{\left[ n\right] }\right) ^{1}\right) \left( x^{\left[ n%
\right] }\right) ^{2}=\delta _{n,0}x^{\left[ n\right] }=\delta _{n,0}x^{%
\left[ 0\right] }=\varepsilon _{R}\left( x^{\left[ n\right] }\right) 1_{R}.
\end{equation*}%
On the other hand we have%
\begin{eqnarray*}
&&
\left( x^{\left[ n\right] }\right) ^{1}S_{R}\left( \left( x^{\left[ n%
\right] }\right) ^{2}\right) \\
&=&
\sum_{0\leq i\leq n}\beta (i,n)x^{[i]}S_{R}\left( x^{[n-i]}\right) \\
&=&
\sum_{0\leq w\leq n}\beta (n-w,n)x^{[n-w]}S_{R}\left( x^{[w]}\right) \\
&=&
\sum_{0\leq w\leq n}\beta (n-w,n)\left( -1\right) ^{w}\chi \left(
g\right) ^{\frac{w\left( w-1\right) }{2}}x^{[n-w]}x^{[w]} \\
&\overset{\eqref{eqn: beta}}{=}&
\sum_{0\leq w\leq n}\binom{n}{w}_{q}\left( -1\right) ^{w}q^{\frac{w\left(
w-1\right) }{2}}x^{\left[ n\right] }=\delta _{n,0}x^{\left[ n\right]
}=\delta _{n,0}x^{\left[ 0\right] }=\varepsilon _{R}\left( x^{\left[ n\right]
}\right) 1_{R}.
\end{eqnarray*}%
We note that $S_{R}:R\rightarrow R$ is trivially bijective.
\end{proof}

Recall the definition of a morphism of quasi-$YD$ data from Definition \ref{def: morphism YD data}.
Note that if $\varphi:((H,\omega), g, \chi) \rightarrow ((L,\alpha), \ell, \xi) $ is a morphism of quasi-$YD$ data,
 with $((H,\omega), g,\chi)$  a quasi-$YD$ datum for $q$ then
   $((L,\alpha),\ell, \xi)$ is also a quasi-$YD$ datum for $q$
  since $  \xi(\ell) = \xi \varphi(g)= \chi(g)$.
  It follows easily from equation (\ref{form:chin}) that $\xi_{[n]}  \varphi = \chi_{[n]}$ for all $n \geq 1$.
  The proof of the next proposition is straightforward and so the details are left to the reader.

\begin{proposition}
\label{pro:morphdatum}Let $\varphi :\left( \left( H,\omega \right) ,g,\chi
\right) \rightarrow \left( \left( L,\alpha \right) ,l,\xi \right) $ be a
morphism of quasi-$YD$ data with $q:=\chi \left( g\right) $
a primitive $N$-th root of unity , $N>0$. Let $\left( x^{\left[ n\right]
}\right) _{0\leq n\leq N-1}$ be the canonical basis for $R_{H}:=R\left(
\left( H,\omega \right) ,g,\chi \right) $ and   $\left( y^{\left[ n\right]
}\right) _{0\leq n\leq N-1}$  the canonical basis for $R_{L}:=R\left(
\left( L,\alpha \right) ,l,\xi \right) .$ Consider the $\Bbbk $-linear
isomorphism $f:R_{H}\rightarrow R_{L}$ mapping $x^{\left[ n\right] }$ to $y^{%
\left[ n\right] }$ for all $n\in \left\{ 0,\ldots ,N-1\right\} .$ Then%
\begin{gather*}
\rho
_{R_{L}}  f = \left( \varphi \otimes f\right)  \rho _{R_{H}},\qquad
\mu _{R_{L}}  \left( \varphi \otimes f\right) =f  \mu _{R_{H}}, \\
1_{R_{L}} =f\left( 1_{R_{H}}\right) ,\qquad m_{R_{L}}  \left( f\otimes
f\right) =f m_{R_{H}}, \\
\Delta _{R_{L}}  f =\left( f\otimes f\right)   \Delta _{R},\qquad
\varepsilon _{R_{L}}  f=\varepsilon _{R_{H}}.
\end{gather*}%
Moreover $f\otimes \varphi :R_{H}\#H\rightarrow R_{L}\#L$ is a dual
quasi-bialgebra homomorphism.
\end{proposition}

\section{Quasi-Yetter-Drinfeld data for bosonizations}\label{sec: qYD for bosonizations}

In this section we consider quasi-$YD$ data  for bosonizations $R \# H$.
In the next lemma, we assume that we have a bosonization $B= R \# H$ with a quasi-$YD$ datum and we find
that this yields a quasi-$YD$ datum for $H$.

\begin{lemma}
\label{lem:qYdOnSmash} For $(H, \omega)$   a dual quasi-bialgebra and   $R$   a
bialgebra in ${_{H}^{H}\mathcal{YD}}$, consider the dual quasi-bialgebra $
B:=R\#H$, the bosonization of $R$ by $H$, defined in Theorem \ref{teo:RsmashH}.
 Assume that $B_{0}\subseteq
\Bbbk 1_{R}\otimes H$ and let $\left( (B, \omega_B),g,\chi _{B}\right) $ be a
quasi-$YD$ datum. Then there exists $c\in G\left( H\right) $ such
that $g=1_{R}\# c$, and   $\left( (H,\omega),c,\chi _{B} \sigma \right) $ is
a quasi-$YD$ datum where $\sigma: H \hookrightarrow B$ is the inclusion.
 Moreover, for every $r\in R,h\in H$ we have%
\begin{eqnarray}
\chi _{B}\left( r\# h\right) &=&\omega _{H}^{-1}\left( r_{-2}\otimes
h_{1}\otimes c\right) \chi _{B}\left( 1 \# h_{2}\right) \omega
_{H}\left( r_{-1}\otimes c\otimes h_{3}\right) \chi _{B}\left( r_{0}\#
1 \right) ,  \label{form:Chi1}
\\
\chi _{B}\left( r\cdot _{R}s\# 1 \right) &=&\omega _{H}^{-1}\left(
r_{-1}\otimes s_{-1}\otimes c\right) \chi _{B}\left( s_{0}\#
1_{H}\right) \chi _{B}\left( r_{0}\# 1 \right) ,  \label{form:Chi2}
\\
\chi _{B}\left( r\# h_{1}\right) ch_{2} &=&\left( r_{-1}h_{1}\right)
\chi _{B}\left( r_{0}\# h_{2}\right) c,  \label{form:iter1}
\\
\chi _{B}\left( r^{1}\# r_{-1}^{2}\right) c\vartriangleright r_{0}^{2}
&=&\omega _{H}^{-1}(r_{-1}^{1}\otimes r_{-1}^{2}\otimes c)r_{0}^{1}\chi
_{B}\left( r_{0}^{2}\# 1_{H}\right) .  \label{form:iter2}
\end{eqnarray}%

\end{lemma}

\begin{proof}
   Since $g \in G\left( B\right) \subseteq
B_{0}\subseteq \Bbbk \otimes H$, then   $
g=1_{R}\# c = \sigma(c)$ for some $c \in H$.  Thus $c =( \pi \sigma)(c) = \pi(g)$ and since the maps $\pi,\sigma$
from Section \ref{subsec: YD modules} are coalgebra maps and $g$ is grouplike, then $c$ is grouplike.
\par  In order to apply Lemma \ref{lem:qYDquotient} to conclude that $\left( (H,\omega),c,\chi _{B} \sigma \right) $
is a quasi-$YD$ datum, we must show that $c \chi_B \sigma(h_1) h_2 = h_1 \chi_B \sigma(h_2)c$ for all $h \in H$.
Since $((B, \omega_B), g, \chi_B)$ is a quasi-$YD$ datum, and so satisfies
  (\ref{form: YDdatum}), then   for every $r\in R,h\in H$,
\begin{equation*}
  g\chi _{B}\left( \left( r\#h\right) _{1}\right) \left( r\#h\right)
_{2}=\left( r\#h\right) _{1}\chi _{B}\left( \left( r\#h\right) _{2}\right) g.
\end{equation*}%

 If we let $r = 1_R$ in the equation above, and apply $\pi$ to both sides, we obtain
\begin{equation*}
c \chi_B(\sigma(h)_1) \pi(\sigma(h)_2) = \pi(\sigma(h)_1) \chi_B(\sigma(h)_2)c,
\end{equation*}
and since $\sigma, \pi$ are   coalgebra maps with $\pi \sigma$ the identity,
 then (\ref{form: YDdatum}) holds for $((H,\omega),c,\chi_B \sigma)$ and
by  Lemma \ref{lem:qYDquotient}, $((H,\omega),c,\chi_B \sigma)$ is a quasi-$YD$ datum.

\par Since $\omega _{B}=\omega _{H} \circ \pi^{ \otimes 3}
  $ then for all $x,y\in B $,  by (\ref{form:chiProd}) for the quasi-$YD$ datum for $B$,
we have that $\chi _{B}\left( xy\right)$ is:
\begin{eqnarray}
  &&  \label{form:chiProd-reduced}
   \omega
_{B}^{-1}\left( x_{1}\otimes y_{1}\otimes g\right) \chi _{B}\left(
y_{2}\right) \omega _{B}\left( x_{2}\otimes g\otimes y_{3}\right) \chi
_{B}\left( x_{3}\right) \omega _{B}^{-1}\left( g\otimes x_{4}\otimes
y_{4}\right)
\\
&=&\omega _{H}^{-1}\left( \pi \left( x_{1}\right) \otimes \pi \left(
y_{1}\right) \otimes c \right) \chi _{B}\left(
y_{2}\right) \omega _{H}\left( \pi \left( x_{2}\right) \otimes c \otimes \pi \left( y_{3}\right) \right) \chi _{B}\left(
x_{3}\right) \omega _{H}^{-1}\left( c \otimes \pi \left(
x_{4}\right) \otimes \pi \left( y_{4}\right) \right) \nonumber 
\end{eqnarray}%

By (\ref{form:piIdpi}),
\begin{eqnarray}
 (\pi \otimes \pi \otimes B \otimes \pi) \Delta_B^3 (r \# 1) &=&
  r_{- 2} \otimes r_{-1} \otimes (r_0 \# 1 ) \otimes 1_H; \label{form: (13) for r times 1}\\
 (\pi \otimes B \otimes \pi \otimes \pi) \Delta_B^3 (1 \# h) &=&
 h_1 \otimes (1 \# h_2) \otimes h_3 \otimes h_4 \label{form: (13) for 1  times h},
\end{eqnarray}
and so, letting $x = r \# 1$ and $y = 1 \# h$, we have that $\chi_B(r \# h)$ is:
\begin{equation}
\omega _{H}^{-1}\left(   r_{-2}  \otimes h_1   \otimes c \right) \chi _{B}\left(
1 \# h_2 \right) \omega _{H}\left(r_{-1} \otimes c \otimes h_3   \right) \chi _{B}\left(
r_0 \# 1 \right) \omega _{H}^{-1}\left( c \otimes 1 \otimes h_4 \right),
\end{equation}
and since $\omega_H$ is normalized, (\ref{form:Chi1}) holds.

\par  Similarly $\chi_B(r \cdot_R s \# 1) = \chi_B((r \# 1)(s \# 1))$ and then, using
  (\ref{form: (13) for r times 1}) and (\ref{form:chiProd-reduced}), along with
  \begin{equation*}
  (\pi \otimes B \otimes \pi \otimes \pi) \Delta^3_B(s \# 1_H) = s_{-1} \otimes (s_0 \# 1) \otimes 1_H \otimes 1_H,
  \end{equation*}
  it is straightforward to verify  (\ref{form:Chi2}).

  \par Now we prove (\ref{form:iter1}).  Since $((B,\omega_B), g, \chi_B)$ satisfies (\ref{form: YDdatum}),  we have,
\begin{equation*}
\text{ }g\chi _{B}\left( \left( r\#h\right) _{1}\right) \left( r\#h\right)
_{2}=\left( r\#h\right) _{1}\chi _{B}\left( \left( r\#h\right) _{2}\right) g,%
\text{ for every }r\in R,h\in H.
\end{equation*}%
Recall from   Theorem \ref{teo:RsmashH} that
\begin{equation*}
\Delta_B(r \# h) = \omega_H^{-1}(r_{-1}^1 \otimes r_{-2}^2 \otimes h_1) r_0^1 \#
r^2_{-1} h_2 \otimes r_0^2 \# h_3,
\end{equation*}
so that applying $\pi$ to the left hand side of (\ref{form: YDdatum}) for $B$ we obtain:
\begin{equation*}
c \chi_B(r_0^1 \# r^2_{-1}h_2)\omega^{-1}_H( r_{-1}^1 \otimes r_{-2}^2 \otimes h_1) \varepsilon(r_0^2) h_3
= \chi_B(r \# h_1)ch_2.
\end{equation*}
Applying $\pi$ to the right hand side yields
\begin{equation*}
\omega_H^{-1}(r_{-1}^1 \otimes r_{-2}^2 \otimes h_1) \varepsilon(r_0^1)r_{-1}^2 h_2 \chi_B(r_0^2 \# h_3)c
= \chi_B(r_0 \# h_2)r_{-1}h_1 c,
\end{equation*}
and thus (\ref{form:iter1}) holds.

\par Equation (\ref{form:iter2}) is verified in a similar fashion.  Let $h = 1$ in the
left hand side of equation
(\ref{form: YDdatum}) for $B$ and then apply $R \otimes \varepsilon_H$ to obtain
\begin{equation*}
\chi_B((r \#1)_1)(R \otimes \varepsilon_H) [(1 \#c)(r \#1)_2]   =
\chi_B(r_0^1 \# r_{-1}^2)(R \otimes \varepsilon_H) [ c\triangleright r_0^2 \#c]
= \chi_B(r^1 \otimes r_{-1}^2)c\triangleright r_0^2.
\end{equation*}
Now let $h=1$  in the right hand side of (\ref{form: YDdatum}) for $B$ and apply $R \otimes \varepsilon_H$ to obtain
\begin{eqnarray*}
&&\chi_B((r \#1)_2) (R \otimes \varepsilon_H)[(r \#1)_1 (1 \#c)]
 = \chi_B(r_0^2 \#1) (R \otimes \varepsilon_H)[(r_0^1 \# r^2_{-1})(1 \#c)]  \\
&& = \chi_B(r_0^2 \#1)\omega^{-1}_H( (r_0^1)_{-1} \otimes (r^2_{-1})_1 \otimes c) (r_0^1)_0 \varepsilon_H((r^2_{-1})_2c)\\
&& = \chi_B(r_0^2 \#1)\omega^{-1}_H( r^1_{-1} \otimes r^2_{-1} \otimes c) r_0^1,
\end{eqnarray*}
and this finishes the proof of (\ref{form:iter2}).
\end{proof}

In the  next proposition we show how an arbitrary quasi-$YD$ datum on a bosonization
$R \#H$ where $R:= R((H,\omega_H),g_H,\chi_H)$ is related to $g_H$ and $\chi_H$.

\begin{proposition}
\label{pro:qydd}Let   $\left( \left( H,\omega _{H}\right) ,g_{H},\chi
_{H}\right) $ be a quasi-$YD$ datum for  a primitive $N$-th
root of unity $q$
 and let $R=R\left( \left( H,\omega _{H}\right) ,g_{H},\chi
_{H}\right) $ be the bialgebra in ${_{H}^{H}\mathcal{YD}}$ introduced in
Theorem \ref{teo:R}. Let $B=R\#H$, the bosonization of $R$ by $H$ and suppose that
 $\left(( B, \omega_B),g_{B},\chi _{B}\right) $ is a
quasi-$YD$ datum. Then there exists $d\in G\left( H\right) $ such
that $g_{B}=1_{R}\# d $. If
 $d\neq g_{H}d$, then
\begin{itemize}
\item[(i)]$ \chi _{B}\left( r\# h\right)= \varepsilon _{R}\left( r\right) \chi
_{B}\left( 1_{R}\# h\right)$ ,  for every $r\in R,h\in H$,
\item[(ii)]$ \chi _{B}\left( 1_{R}\# g_{H}\right) \chi _{H}\left( d\right) =1$,
\item[(iii)]$ dg_{H}= g_{H}d.$
\end{itemize}
\end{proposition}

\begin{proof}
  Theorem \ref{teo:RsmachHthin} implies that $B_{0}\subseteq \Bbbk 1_{R}\otimes
H$. Then Lemma \ref{lem:qYdOnSmash} implies that there exists $d\in
G\left( H\right) $ such that $g_{B}=1_{R}\# d.$ By (\ref{form:iter1})
with $r = x^{[1]}$ and $h = 1_H$,

\begin{equation*}
\chi _{B}\left( x^{\left[ 1\right] }\# 1_{H}\right) d=\chi _{B}\left(
x^{\left[ 1\right] }\# 1_{H}\right) g_{H}d.
\end{equation*}%
If $\chi _{B}\left( x^{\left[ 1\right] }\# 1_{H}\right) \neq 0$, then $d = g_H d$, contrary
to our assumption and so $\chi _{B}\left( x^{\left[ 1\right] }\# 1_{H}\right) = 0$.

Now, let $2\leq n\leq N-1$ and assume $\chi _{B}\left( x^{\left[ n-1\right]
}\# 1_{H}\right) =0.$ Then%
\begin{eqnarray*}
\chi _{B}\left( x^{\left[ n\right] }\# 1_{H}\right) &=&\chi _{B}\left(
x^{\left[ 1\right] }\cdot _{R}x^{\left[ n-1\right] }\# 1_{H}\right)
 \\
&\overset{(\ref{form:Chi2})}{=}&\omega _{H}^{-1}\left( g_{H}\otimes g_{H}^{n-1}\otimes d\right) \chi
_{B}\left( x^{\left[ n-1\right] }\# 1_{H}\right) \chi _{B}\left( x^{%
\left[ 1\right] }\# 1_{H}\right) =0,
\end{eqnarray*}%
so that

\begin{equation*}
\chi _{B}\left( x^{\left[ n\right] }\# 1_{H}\right) =\delta _{n,0},%
\text{ for }0\leq n\leq N-1.
\end{equation*}%
Now%
\begin{eqnarray*}
\chi _{B}\left( x^{\left[ n\right] }\# h\right)
& \overset{(\ref{form:Chi1})}{=}&
\omega _{H}^{-1}\left( g_{H}^{n}\otimes h_{1}\otimes d\right) \chi
_{B}\left( 1_{R}\# h_{2}\right) \omega _{H}\left( g_{H}^{n}\otimes
d\otimes h_{3}\right) \chi _{B}\left( x^{\left[ n\right] }\#
1_{H}\right)
\\
&=& \delta _{n,0}\chi _{B}\left( 1_{R}\# h\right),
\end{eqnarray*}%
and so
\begin{equation*}
\chi _{B}\left( r\# h\right) =\varepsilon _{R}\left( r\right) \chi
_{B}\left( 1_{R}\# h\right) ,\text{ for every }r\in R,h\in H.
\end{equation*}%
Next we consider equation (\ref{form:iter2}) with $r = x^{[1]}$.

The left hand side is%
\begin{eqnarray*}
&&\chi _{B}\left( \left( x^{\left[ 1\right] }\right) ^{1}\# \left( x^{%
\left[ 1\right] }\right) _{-1}^{2}\right) d\vartriangleright \left( x^{\left[
1\right] }\right) _{0}^{2}
\\
&=&\chi _{B}\left( x^{\left[ 1\right] }\# 1_{H}\right)
d\vartriangleright 1_{R}+\chi _{B}\left( 1_{R}\# g_{H}\right)
d\vartriangleright x^{\left[ 1\right] } \\
&=&\chi _{B}\left( 1_{R}\# g_{H}\right) d\vartriangleright x^{\left[ 1%
\right] }=\chi _{B}\left( 1_{R}\# g_{H}\right) \chi _{H}\left( d\right)
x^{\left[ 1\right] }
\end{eqnarray*}%
and the right hand side is%
\begin{eqnarray*}
&&\omega _{H}^{-1}(\left( x^{\left[ 1\right] }\right) _{-1}^{1}\otimes
\left( x^{\left[ 1\right] }\right) _{-1}^{2}\otimes d)\left( x^{\left[ 1%
\right] }\right) _{0}^{1}\chi _{B}\left( \left( x^{\left[ 1\right] }\right)
_{0}^{2}\# 1_{H}\right)
 \\
 &=& \omega_H^{-1}(1_H \otimes g_H \otimes d) \chi_B(x^{[1]} \# 1_H)
 + \omega_H^{-1} (g_H \otimes 1_H \otimes d) x^{[1]} \chi_B(1_R \# 1_H)
  \\
&=&x^{\left[ 1\right] }+1_{R}\chi _{B}\left( x^{\left[ 1\right] }\#
1_{H}\right) =x^{\left[ 1\right] }.
\end{eqnarray*}%
and we can conclude that
\begin{equation*}
\chi _{B}\left( 1_{R}\# g_{H}\right) \chi _{H}\left( d\right) =1.
\end{equation*}%

Now we apply (\ref{form: YDdatum}) for the quasi-$YD$ datum
  $\left( H,d,\chi _{B}\left( 1_{R}\otimes -\right) \right) $ from Lemma \ref{lem:qYdOnSmash} with $h = g_H$ to obtain
\begin{equation*}
\chi _{B}\left( 1_{R}\# g_{H}\right) dg_{H}=g_{H}d\chi _{B}\left(
1_{R}\# g_{H}\right) ,
\end{equation*}%
and since $\chi_B(1_R \# g_H)$ is invertible, we obtain
\begin{equation*}
dg_{H}=g_{H}d.
\end{equation*}
\end{proof}

\section{Examples}

In this section, we present examples illustrating the theory in the previous sections. The problem of course is to find the reassociator explicitly.  Our examples are based on  the  coalgebra $\Bbbk C_n$ where the cocycles are well-known.

\subsection{Group cohomology}\label{sec: group cohomology}

First we set some notation.  Our examples   will involve cyclic groups of order $n$ and $n^2$, $n>1$.
We will denote $C_n = \langle c \rangle$ and  $C_{n^2} = \langle \mathfrak{c} \rangle$.  We will always denote by $q$ a
primitive $n^{2}$-rd root of unity and set $\zeta:=q^n $ . Let $\phi : C_{n^2}\rightarrow C_n$ be the canonical projection with $\phi(\c)=c$ and denote by the same symbol the corresponding map $\Bbbk C_{n^2}\rightarrow \Bbbk C_n$.
For every $a\in\mathbb{Z}$, let $a^{\prime }\in \{0,\ldots, n-1\}$ be congruent to $a$ modulo $n$.

\par Since $\Bbbk $ is an algebraically closed field of characteristic zero, by \cite[Theorem 3.1]{Sweedler-Cohom}, the Sweedler  cohomology can be
computed through an isomorphism%
\begin{equation*}
H_{sw}^{t}\left( \Bbbk C_{n},\Bbbk \right) \cong H^{t}\left( C_{n},\Bbbk
^{\times }\right) ,
\end{equation*}%
where the latter is the group cohomology computed as in \cite[page 167]{Weibel}.

\par For $0\leq i\leq
n-1 $ and $0 \leq a,b,d  $, define  $\omega _{\zeta ^{i}}:(\Bbbk  C_{n})^{ \otimes 3}
   \rightarrow \Bbbk $
by
\begin{equation}\label{eqn: defn of omega}
\omega _{\zeta ^{i}}\left( c^{a}\otimes  c^{b}\otimes  c^{d}\right) =
\omega _{\zeta ^{i}}\left( c^{a}\otimes  c^{d}\otimes  c^{b}\right)
= \zeta^{ia[[b^\prime +d^\prime > n-1]]}.
\end{equation}%
%
Since $\zeta = q^n$,
it is easy to check that
\begin{equation}\label{omegaq}
\omega _{\zeta ^{i}}\left( c^{a}\otimes c^{b}\otimes c^{d}\right) = \zeta ^{ia[[b^\prime +d^\prime >n-1  ]]}
 = q^{ina[[b^\prime +d^\prime > n-1  ]]}  \\
=  q^{ia\left( b^\prime +d^\prime -\left( b+d\right) ^{\prime }\right) }.
\end{equation}%

One can  prove that the set of Sweedler $3$%
-cocycles is given by
\begin{equation}\label{form: Sweedler 3 cocycles}
Z_{sw}^{3}\left( \Bbbk C_{n},\Bbbk \right) =\left\{ \left( \omega _{\zeta
^{i}}\right) ^{v} = \omega_{\zeta^i}\ast \partial^2v \mid 0\leq i\leq n-1,v:
\Bbbk  C_{n}^{ \otimes 2}
 \rightarrow \Bbbk \text{ is convolution invertible}%
\right\}.
\end{equation}%
This follows from the fact (see e.g. \cite[formulas (E.13) and (E.14)]{MS} over $\mathbb{C}$) that the map
\begin{equation}\label{Sw}
\left\{ k\in \Bbbk ^{\times }\mid k^{n}=1\right\} \rightarrow H_{sw}^{3}\left(
\Bbbk C_{n},\Bbbk \right) :k\mapsto \left[ \omega _{k}\right]
\end{equation}
is a group isomorphism.%

\begin{proposition}
\label{pro:KCn} Let $ \Bbbk   C_{n}
 $   be the group algebra with its standard bialgebra structure and   $\omega$
 a  normalized  $3$-cocycle. Then  $(\Bbbk C_n, \omega)$ is a dual quasi-bialgebra
and there is a gauge transformation $\alpha :(\Bbbk
C_{n})^{\otimes 2}   \rightarrow \Bbbk $,
and $0\leq i\leq n-1$ such that $(\Bbbk C_n, \omega) = (\Bbbk C_n, \omega_{\zeta^i})^\alpha
= (\Bbbk C_n, \omega_{\zeta^i}\ast \partial^2\alpha).$
\end{proposition}

\begin{proof}The first statement follows from the fact that $\Bbbk C_n$ is cocommutative.
     Since $\omega $ is a  normalized  Sweedler $3$-cocycle, by (\ref{form: Sweedler 3 cocycles}) there
exists a convolution invertible map $v:\Bbbk   C_{n}^{  \otimes 2}
  \rightarrow \Bbbk $  and $i\in \left\{ 0,\ldots
,n-1\right\} $ such that $\omega =\left( \omega _{\zeta ^{i}}\right) ^{v} = \partial^2v \ast \omega_{\zeta^i} $
and $(\Bbbk C_n, \omega) = (\Bbbk C_n, \omega _{\zeta ^{i}}  ^{v}) =  (\Bbbk C_n, \omega _{\zeta ^{i}} ) ^{v}$.  Since $\omega$ and $\omega_{\zeta^i}$ are normalized, so is $\partial^2 v$.
Thus, by Corollary \ref{co: delta2vnormal}, $av$ is a gauge transformation for  $a = v(1 \otimes 1)^{-1}$.
Since $(\Bbbk C_n, \omega_{\zeta^i})^v = (\Bbbk C_n, \omega_{\zeta^i})^{av}$, the statement is proved.
\end{proof}

\begin{remark}In fact, $(\Bbbk C_n, \omega_{\zeta^i})$ is a dual quasi-Hopf algebra, meaning that
  there exists an antipode $S$ and maps $\alpha, \beta$ from $\Bbbk C_n$ to $\Bbbk$ such that for all $h \in \Bbbk C_n$:
  \begin{eqnarray}
&&  S(h_1)\alpha(h_2)h_3 = \alpha(h)1 \text{  and } h_1\beta(h_2)S(h_3) = \beta(h)1;\\
 &&  \omega_{\zeta^i}(h_1\beta(h_2)\otimes S(h_3)\otimes \alpha(h_4)h_5)
   = \omega_{\zeta^i}^{-1}(S(h_1)\otimes \alpha(h_2)h_3 \otimes \beta(h_4)S(h_5)) = \varepsilon(h).
  \end{eqnarray}

  In this case, $S$ is the usual antipode for $\Bbbk C_n$, $\alpha$ and $\beta$
are both equal to the counit $\varepsilon$ and then since $\omega_{\zeta^i}(c^j \otimes c^{-j} \otimes c^j) = 1$,
the statement follows.
\end{remark}


Since by the above discussion the maps $\omega_{\zeta^i}$ are not coboundaries, we have the following:

\begin{corollary}\label{cor: not quasi iso to trivial} The dual quasi-bialgebra $(\Bbbk C_n, \omega_{\zeta^i})$
 is not
quasi-isomorphic
to an ordinary bialgebra, i.e., one with reassociator $\varepsilon_{\Bbbk C_n^{\otimes 3} }$.
\end{corollary}

On the other hand $\Bbbk C_{n^2}$ with dual quasi-bialgebra structure via the bialgebra epimorphism $\phi: \Bbbk C_{n^2} \rightarrow \Bbbk C_n$, $\phi(\c) = c$,  is quasi-isomorphic  to an ordinary bialgebra since if
  $\omega$ is a normalized $3$-cocycle for $\Bbbk C_n$, then $\omega   \phi^{\otimes 3}$ is   a coboundary.
In fact,  one can see by direct computation that
 $\omega_{\zeta^i} \phi^{\otimes 3} = \partial^2 v_i$ where
$v_i: (\Bbbk C_{n^2})^{\otimes 2} \rightarrow \Bbbk$ is defined by $v_i(\c^a \otimes \c^b) = q^{ia(b-b^\prime)}$.

\subsection{Quasi-$YD$ data for $\Bbbk C_n$}

To find quasi-$YD$ data for $(\Bbbk C_n, \omega_{\zeta^w})$, we apply the results of Section \ref{sec: qyd data}.
For $0 \leq z \leq n-1$, from (\ref{form:datum-1 short}) we will be able to show
    that if  $((\Bbbk C_n, \omega_{\zeta^w}), g:= c^z, \chi)$ is a quasi-$YD$ datum then
\begin{equation} \label{form:datum-0 special}
\chi(c^t) = \chi(c)^t  \text{ for }  0 \leq t \leq n-1 \text{ and }  \chi(c)^n =\zeta^{wz},
\end{equation}
and
thus, unless $\chi(c)^n = 1$, i.e., $\zeta^{wz} =1$, then $\chi$ is not a character.

  We show  \eqref{form:datum-0 special} as follows. From \eqref{form:datum-1 short} and the definition of $\omega_{\zeta^w}$, we deduce that $\chi(c^t) = \chi(c)^t$  for $1 \leq t \leq n-1$. By unitarity of $\chi$, this equality also holds for $t=0$. By unitarity of $\chi$ and the fact that $c^n=1$, we get $1 = \chi(1) = \chi(c^n)$. On the other hand a direct computation of $\chi(c^{n})$ using \eqref{form:datum-1 short} and the definition of $\omega_{\zeta^w}$ yields $\chi(c^{n})=\chi(c)^n \zeta^{-wz}$ and so \eqref{form:datum-0 special} is proved.

Take $t\in\mathbb{N}$.
Then, since $c^n=1$, we have $\chi(c^{t})=\chi(c^{t^\prime})\overset{\eqref{form:datum-0 special}}{=}\chi(c)^{t' }$.   Thus \eqref{form:datum-0 special} is equivalent to
\begin{equation}\label{form:datum-1 special}
\chi(c^t) = \chi(c)^{ t^\prime}\text{ for } t\in\mathbb{N}   \text{ and } \chi(c)^n  =  \zeta^{wz} .
\end{equation}
\end{remark}

%
%

\begin{proposition}
\label{pro:speriamo} Consider the dual quasi-bialgebra $(\Bbbk C_n, \omega_{\zeta^w})$.
 Let $c^z \in C_n$, $0 \leq z \leq n-1$ and $\chi \in \Bbbk C_n^\ast$.
  If $((\Bbbk C_n, {  \omega_{\zeta^w}}), c^z, \chi)$ is a quasi-$YD$ datum,
   then (\ref{form:datum-1 special}), or equivalently (\ref{form:datum-0 special}),
    holds.
Conversely  if $\chi$ is a unitary map satisfying
{  \eqref{form:datum-1 special}},  or equivalently (\ref{form:datum-0 special}),
  for some $0 \leq z \leq n-1$ then $((\Bbbk C_n, \omega_{\zeta^w}),c^z,\chi)$ is a
quasi-$YD$ datum.

\end{proposition}

\begin{proof}
The first assertion follows immediately from
 Lemma \ref{lem:ChiCyclic} and Remark \ref{rem: datum short}.

\par Since $c^z \in G\left( \Bbbk C_n\right) $, since $\chi \in \Bbbk C_n^{\ast }$ is unitary
by assumption, and since (\ref{form: YDdatum}) holds because $\Bbbk C_n$ is both commutative and cocommutative,
it remains only to check (\ref{form:chiProd}) .
We   check (\ref{form:chiProd}) on generators. The equality
holds trivially for $h=1_{\Bbbk C_n}$ or $k=1_{\Bbbk C_n}.$ Hence we can assume that $%
h=c^{a}$ and $k=c^{b}$ for $1 \leq a,b \leq n-1.$    Then the left side of (\ref{form:chiProd}) is %
 \begin{equation*}
\chi \left( c^{a}c^{b}\right) = \chi \left(
c^{a+b}\right)
\overset{\eqref{form:datum-1 special}}{=}
\chi \left( c\right) ^{\left( a+b\right) ^{\prime }}
= \chi \left( c\right) ^{a+b-[[a+b \geq n]]n}
\overset{(\ref{form:datum-1 special})}{=}\chi \left( c\right) ^{a+b}\zeta ^{-[[a+b \geq n]]wz}.
\end{equation*}


Since $\omega_{\zeta^w} = \omega_{\zeta^w}(\Bbbk C_n \otimes \tau)$ and $\Bbbk C_n$ is cocommutative, the right hand side of (\ref{form:chiProd}) is:
  \begin{equation*}
 \omega _{\zeta ^{w}}^{-1}\left( c^{z}\otimes c^{a}\otimes c^{b}\right)
\chi \left( c^{a}\right) \chi \left( c^{b}\right) \\
 \overset{(\ref{form:datum-1 special})}{=}
\chi \left( c\right) ^{a+b}\omega _{\zeta ^{w}}^{-1}\left( c^{z}\otimes
c^{a}\otimes c^{b}\right)\overset{(\ref{eqn: defn of omega})}{=}\chi \left( c\right) ^{a+b}\zeta ^{-[[a+b \geq n]]wz}.
\end{equation*}
Thus (\ref{form:chiProd}) holds and the proof is complete.
 \end{proof}

\begin{example}\label{ex: basic example qYD datum} Consider the dual quasi-bialgebra $(\Bbbk C_n, \omega_{\zeta^i})$ with $i>0$.
 Let $c^z \in C_n$ with $1 \leq z \leq n-1$, and then for $\chi \in (\Bbbk C_n)^\ast$ to satisfy \eqref{form:datum-0 special},
 we must have that $\chi(c)^n = \zeta^{iz}
 = q^{niz}$.  Thus if we define $\chi_{}(c^t) = q^{izt}$ for $0 \leq t \leq n-1$, then $((\Bbbk C_n, \omega_{\zeta^i}), c^z, \chi_{})$
 is a quasi $YD$-datum for $\chi(c^z) = q^{iz^2} $.  Note that $q^{iz^2}$ is a primitive $r$th root of unity
 where $r = \frac{n^2}{(n^2,iz^2)}$.

 \par
 More generally, for $0 \leq j \leq n-1$, let $\chi_j: \Bbbk C_n \rightarrow \Bbbk$ be defined by
 $\chi_j(c^t) = \zeta^{jt}q^{izt}$ if $0 \leq t \leq n-1$. Since $\chi_j(c)^n = (\zeta^{j})^n q^{izn} = \zeta^{iz}$, so that
  \eqref{form:datum-0 special} holds,
 $((\Bbbk C_n, \omega_{\zeta^i}), c^z, \chi_{j})$ is   a quasi-$YD$ datum for $\chi_j(c^z) = \zeta^{jz}q^{iz^2}$
 and the order of $\zeta^{jz}q^{iz^2}$ is $\frac{n^2}{(njz + iz^2, n^2)}$.

\end{example}

\begin{example} \label{ex: p(p-1) sq boson}  Let $n=p$, a prime.  For the quasi-$YD$ datum
$((\Bbbk C_p, \omega_{\zeta^i}), c^z, \chi_j)$ in Example \ref{ex: basic example qYD datum}
with $\chi_j(c^t) = \zeta^{jz}q^{izt}$,
 there are $p-1$ choices
for $i$ and  also for $z$ that give  quasi-$YD$ data for a primitive $p^2$rd root of unity
and since
  $j= 0, \ldots, p-1$, there are $p$ choices for $j$.
Thus one may
form $p(p-1)^2$ bosonizations $R \# \Bbbk C_p$ where $R$ has dimension $p^2$.  Below we discuss which of these can
be isomorphic or quasi-isomorphic.

\par   Suppose that $H:= (\Bbbk C_p, \omega_{\zeta^i})$ and $L:= (\Bbbk C_p , \omega_{\zeta^{i^\prime}})$.  Then by
the discussion in Section \ref{sec: group cohomology}, $H $ and $L$ are quasi-isomorphic if and only
if $i = i^\prime$.
If $R \#H$  is quasi-isomorphic   to $S \#L$ for some $R,S$
 as in Theorem \ref{teo:RsmachHthin}, then,
  by Remark \ref{rem:quasi}, $H$
 is quasi-isomorphic  to
$L$ and thus $i = i^\prime$.
 Thus if two bosonizations as constructed above are quasi-isomorphic, then $i = i^\prime$, i.e., $H=L$.

\par  Now fix $H:= (\Bbbk C_p, \omega_{\zeta^i})$, and consider the quasi-$YD$ data $\mathcal{D} :=(H, c^z, \chi_j)$ and
$\mathcal{E}:=(H,c^w, \chi_k)$.    Let $R$ (respectively $S$) be the Hopf algebra in $^H_H\mathcal{YD}$
constructed from $\mathcal{D}$ ($\mathcal{E}$ respectively) with basis $x^{[n]}$ (respectively $y^{[n]}$).
Set $x=x^{[1]},y=y^{[1]}$.
 Suppose that there
is a dual quasi-bialgebra isomorphism $\Phi: R\#H \rightarrow S \#H$.
By the formula  for $\Delta_R$ in Theorem \ref{teo:RsmachHthin}, the comultiplication formula from Theorem \ref{teo:RsmashH} and the fact that the coefficients
$\beta(i,n)$ are nonzero, we have that  $P_{1\#1,1\#c^{i}}\left( R\#H\right) =\Bbbk 1\#\left( 1-c^{i}\right) +\delta
_{i,z}\Bbbk \left( x \#1\right)$ (respectively $P_{1\#1,1\#c^{i}}\left( S\#H\right) =\Bbbk 1\#\left( 1-c^{i}\right) +\delta
_{i,w}\Bbbk \left( y \#1\right)$).

  Since $\Phi$ is a morphism of dual quasi-bialgebras, by Remark \ref{rem:quasi}, we get that $\Phi(1 \#c^z) =1\#\Phi^\prime(c^z)$.
Write $\Phi^\prime(c^z)=c^a$ with $0\leq a\leq p-1$.
Since $x \#1\in P_{1\#1,1\#c^{z}}\left( R\#H\right)$, we get that
$\Phi(x \#1)\in P_{1\#1,1\#c^a}\left( S\#H\right)=\Bbbk 1\#\left( 1-c^a\right) +\delta
_{a,w}\Bbbk \left( y \#1\right)$. If $a\neq w$, then $\Phi(x \#1)\in \Bbbk 1\#\left( 1-c^a\right)$
and hence $x \#1\in \Bbbk\Phi^{-1}(1\#\left( 1-c^a\right))\subseteq \Bbbk \#H$, a contradiction.  Thus $a=w$ and hence   $\Phi(1 \#c^z) = 1 \#c^w$,  and
$\Phi(x \#1 ) = \alpha y\#1 + \beta 1\#(1-c^w)$.
Since $\Phi^{-1}( 1 \# H) = 1 \#H$ then $\alpha \neq 0$.

 Then we have
\begin{eqnarray*}
\Phi[(1 \#c^z)(x \#1)] &=& \Phi[\chi_j(c^z)x \#c^z]\\
&=& \Phi[\chi_j(c^z)(x \# 1)(1 \# c^z)] \\
&=& \chi_j(c^z) [\alpha y\#1 + \beta 1\#(1-c^w)][1 \#c^w]\\
&=& \chi_j(c^z) [\alpha y\#c^w + \beta 1\#(1-c^w)c^w].
\end{eqnarray*}

However,
\begin{eqnarray*}
\Phi(1 \#c^z)\Phi(x \#1) &=& (1 \# c^w) (\alpha y\#1 + \beta 1\#(1-c^w))\\
&=& \chi_k(c^w)\alpha y \#c^w +   \beta 1 \# (1 - c^w)c^w.
\end{eqnarray*}
Thus $\beta = 0$ and $\chi_j(c^z) = \chi_k(c^w)$, i.e., $\zeta^{jz}q^{iz^2} = \zeta^{kw}q^{iw^2}$.
Thus $p^2$ divides $p(jz-kw) + i(z^2 -w^2)$ so that $p$ divides $i(z-w)(z+w)$.  Then either
$z=w$ or $z+w =p$.

\par Suppose that $z=w$.  Then $p$ divides $z(j-k)$.  This is impossible unless $j=k$ and then the two
quasi-$YD$ data are the same.

\par Suppose that $z+w =p$.  Then $p$ divides $ jz -kw +i(z-w) = jz -k(p-z) + i(2z-p)$ so that
$p$ divides $z(j +k +2i)$, i.e., $p |(j + k + 2i)$.

\par In any case, there are at least $p(p-1)$ nonisomorphic bosonizations.  Fix $z=1$.  Then there are $p$ choices
for $j$ and $p-1$ choices for $i$ giving nonisomorphic bosonizations.
\end{example}

In the next example, for a change, we consider the group algebra of a nonabelian group and find a quasi-$YD$ datum.

\begin{example}\label{ex: quasi YD for dicyclic group} Let $G:= Dic_p$, the dicyclic group of
order $4p$ for $p$ an odd prime.  Then
$Dic_p = C_p \rtimes C_4 = \langle x,y | x^4 = 1 = y^p, xyx^{-1} = y^{-1} \rangle $ and $Z(G) = \{ 1, x^2  \}$.
 Since $C_p$ is a normal
subgroup of $G$ then there is a bialgebra projection $\pi$ from $\Bbbk G $ to $\Bbbk C_4= \Bbbk \langle c \rangle$
 by $\pi(y^ix^j) = c^{j}$. Let $ \omega :=\omega_\zeta$ be the cocycle defined in
 Subsection \ref{sec: group cohomology} for $\Bbbk C_4$ with
  $q$ a primitive $16$th root of unity and $\zeta = q^4$. Let $\omega_G: \Bbbk G^{\otimes 3} \rightarrow \Bbbk$
   be defined by $\omega_G:= \omega \pi^{\otimes 3}$ and thus $(\Bbbk G, \omega_G)$ is a dual quasi-bialgebra and $\pi$ is a
   dual quasi-bialgebra morphism.    By Corollary \ref{coro:subgauge}, since $(\Bbbk C_4, \omega_\zeta)$ is nontrivial and since
   there is an inclusion $\sigma: \Bbbk C_4 \hookrightarrow \Bbbk G$ such that $ \pi \sigma$ is the identity, then
   $(\Bbbk G, \omega_G)$ is also nontrivial.

 \par By Example \ref{ex: basic example qYD datum} with $n=4$, $((\Bbbk C_4, \omega_\zeta), c^2, \chi)$ with $\chi(c^t)
 = q^{2t}$ is a quasi-$YD$ datum for $\chi(c^2) = q^4 = \zeta$, a primitive $4$th root of unity.  By Lemma \ref{lem:qYDquotient},
 since $\pi(x^2) = c^2$ and $x^2 \in Z(G)$, then $((G, \omega_G), x^2, \chi_G:= \chi \pi)$ is a quasi-$YD$ datum for $(\Bbbk G, \omega_G)$.

\end{example}

Note that for a nonabelian group with trivial centre, the construction in the example above can only yield a trivial $YD$
datum for $q=1$.  On the other hand, the same construction is possible for any nonabelian group $G$ with a projection onto a cyclic group such
that the kernel does not contain the centre of $G$.

 \par The next example shows that \eqref{form:datum-0 special} need not hold for a quasi-$YD$ datum for $\Bbbk C_N$ if $\omega \neq \omega_{\zeta^w}$,
 in particular it can happen that $\chi(c^t) \neq \chi(c)^t$ for some $0 < t < N$.

\begin{example}\label{ex: phi} Let $\phi: \Bbbk C_{n^2} = \Bbbk \langle \c\rangle \rightarrow \Bbbk C_n = \Bbbk \langle c \rangle$ be the
surjection of bialgebras from Section \ref{sec: group cohomology} given by $\phi(\c) = c$.  Then $\phi$ induces
a morphism of dual quasi-bialgebras from
 $(\Bbbk C_{n^2}, \omega_\zeta  \phi^{\otimes 3} = \partial^2v)$
to   $(\Bbbk C_n, \omega_\zeta)$  where,
 by Section \ref{sec: group cohomology}, $v(\c^a \otimes \c^b) = q^{a(b - b^\prime)}$.

By Example \ref{ex: basic example qYD datum},  $((\Bbbk C_n, \omega_\zeta), c, \chi)$ is a quasi $YD$-datum,
 with $\chi(c^t) = q^t$ for $0 \leq t \leq n-1$,  and so
by Lemma \ref{lem:qYDquotient}, $((\Bbbk C_{n^2}, \partial^2v), \c, \chi \phi)$ is a quasi-$YD$ datum also.
However, taking $t=n<n^2-1$ and checking \eqref{form:datum-0 special}, we find that
 $\chi \phi (\c^n) = \chi(c^n) = \chi(1) = 1$ while $(\chi\phi (\c))^n =
\chi(c)^n = q^n = \zeta$. Thus in this case \eqref{form:datum-0 special} is not satisfied.
 Note that
 $  \omega_\zeta  \phi^{\otimes 3} =    \omega_\zeta  \phi^{\otimes 3}(\Bbbk C_{n^2} \otimes \tau)$ so that
(\ref{form:datum-1 short}) still holds.
\end{example}

\vspace{2mm}
\begin{example}
Let $\left( \left( H,\omega \right) ,g,\chi \right) $ be a
quasi-$YD$ datum for some primitive $N$-th root of unity $q$, $N>0$.  Let
$L=\Bbbk \left\langle g\right\rangle $ and let $\varphi :(L, \omega_L) \hookrightarrow (H, \omega_H)$ be the
canonical inclusion where $\omega _{L}=\omega _{\mid L^{\otimes 3} }$.
Note that $\varphi :\left( \left( L,\omega _{L}\right) ,g,\chi _{\mid
L}\right) \rightarrow \left( \left( H,\omega \right) ,g,\chi \right) $ is a
morphism of quasi-$YD$ data. By Proposition \ref{pro:morphdatum},
we have a dual quasi-bialgebra homomorphism $f\otimes \varphi
:R_{L}\#L\rightarrow R_{H}\#H$ where $R_{L}:=R\left( \left( L,\omega
_{L}\right) ,g,\chi _{\mid L}\right) $, $R_{H}:=R\left( \left( H,\omega
\right) ,g,\chi \right) $ and $f:R_{L}\rightarrow R_{H}$ is a $\Bbbk $%
-linear isomorphism. Note that, since $\varphi $ is injective, so is $%
f\otimes \varphi $ so that $R_{L}\#L$ identifies with a dual
quasi-subbialgebra of $R_{H}\#H.$
\end{example}

\par We point out that the following example is dual to one given by Gelaki in
\cite[subsection 3.1]{Gelaki}.  There a quasi-Hopf algebra is given which
 is quasi-isomorphic
to an ordinary Hopf algebra but contains a sub-quasi-Hopf algebra which is not.

\begin{example}
\label{ex:Cn}
Recall the setting of Example \ref{ex: phi} where we have a morphism of quasi-$YD$ data
from $((\Bbbk C_{n^2},\omega_{n^2}:=\partial^2v), \c,\chi   \phi )$  to $((\Bbbk C_n, \omega_\zeta),c, \chi   )  $
with $\chi(c^t) = q^t$, $0 \leq t \leq n-1$,
induced
by the bialgebra surjection $\phi: \Bbbk C_{n^2} \rightarrow \Bbbk C_n$ with $\phi(\c) = c$. Note that both are
quasi-$YD$-data for $q$ where $q$ is a primitive $n^2$-rd root of unity.

\par The isomorphism $f: R_{n^{2}} \rightarrow R_{n}$   from  Proposition \ref{pro:morphdatum} yields
  a dual quasi-bialgebra
surjection $f\otimes \phi :R_{n^{2}}\#\Bbbk   C_{n^{2}}
\rightarrow R_{n}\#\Bbbk   C_{n}  $ where $R_{n^{2}}:=R\left(
\left( \Bbbk   C_{n^{2}}  ,\omega _{n^{2} }\right) ,\c,\chi \phi
\right) $, $R_{n}:=R\left( \left( \Bbbk   C_{n}  ,\omega
_{\zeta}\right) ,c,\chi \right) $. Set
\begin{equation*}
A:=R_{n^{2}}\#\Bbbk  C_{n^{2}} \qquad \text{and}\qquad
B:=R_{n}\#\Bbbk  C_{n},
\end{equation*}%
so that $B$ is
  a quotient of the dual quasi-bialgebra
$A.$
 Since $(\Bbbk C_{n^2}, \partial^2v)$ can be twisted by $v^{-1}$ to $(\Bbbk C_{n^2}, \varepsilon_{(\Bbbk C_{n^2})^{\otimes 3}})$,
 then by Remark \ref{rem: with a projection}, $A$ is also quasi-isomorphic to an ordinary bialgebra. In fact, it is easy
 to check that $A$ should be deformed by the gauge transformation $\mu:= v^{-1} (\pi \otimes \pi)$ to obtain an
 ordinary bialgebra. On the other hand,
 since $\omega_\zeta$ is not trivial in $H^3(\Bbbk C_n, \Bbbk)$, $(\Bbbk C_n, \omega_\zeta)$ cannot be quasi-isomorphic
 to an ordinary bialgebra and thus by  Remark \ref{rem: with a projection}, neither can $B$.

\par We can say more about $A^\mu$.  Since $A^\mu$ is finite dimensional with coradical $\Bbbk C_{n^2}$, which is a Hopf algebra,
then $A^\mu$ is also a Hopf algebra  \cite[Remark 36]{Takeuchi-Free}.  Now, let
  $ \left( x^{\left[ n\right] }\right) _{0\leq n\leq N-1}$ be the canonical
basis for $R_{n^{2}}.$ Then  $X:= x^{\left[ 1\right] }\# 1_{\Bbbk   C_{n^{2}}  }$ is  a nontrivial skew-primitive element
    since
\begin{eqnarray*}
\Delta _{A}(x^{\left[ 1\right] }\# 1_{\Bbbk   C_{n^{2}}  })
&=&\left( x^{\left[ 1\right] }\right) ^{1}\# \left( x^{\left[ 1\right]
}\right) _{-1}^{2}\otimes \left( x^{\left[ 1\right] }\right) _{0}^{2}\#
1_{\Bbbk   C_{n^{2}}  }
\\
&=&\left( 1_{R_{n^{2}}}\# \c\right) \otimes \left( x^{\left[ 1\right]
}\# 1_{\Bbbk  C_{n^{2}}  }\right) +\left( x^{\left[ 1\right]
}\# 1_{\Bbbk   C_{n^{2}}  }\right) \otimes \left(
1_{R_{n^{2}}}\# 1_{\Bbbk   C_{n^{2}}  }\right) \\
&=&\left( 1_{R_{n^{2}}}\# \c \right) \otimes \left( x^{\left[ 1\right]
}\# 1_{\Bbbk   C_{n^{2}}  }\right) +\left( x^{\left[ 1\right]
}\# 1_{\Bbbk   C_{n^{2}}  }\right) \otimes 1_{A}
\end{eqnarray*}%
so that, if we set
  $\Gamma :=1_{R_{n^{2}}}\# \c$,  we get
\begin{equation*}
\Delta _{A}(X)=X\otimes 1_{A}+\Gamma \otimes X.
\end{equation*}%
Since $A$ and $A^\mu$ have the same   coalgebra structure,
  $X$ is a $\left( 1_{A},\Gamma \right) $-primitive element also in $
A^{\mu  }$. Consider the sub-Hopf algebra of $A^{\mu  }$
generated by $X$ and $\Gamma .$  This is a Taft algebra of dimension $o\left(
\Gamma \right) ^{2}=n^{4}.$ Hence $A^{\mu  }=T_q$.
\end{example}

\subsection{Quasi-$YD$ data for $R \# \Bbbk C_n$}

Now we apply Proposition \ref{pro:qydd} to a quasi-$YD$ datum used in our examples.
\begin{proposition} \label{prop:B} Let
 $H:=((\Bbbk C_n, \omega_{\zeta }), c , \chi)$ with $\chi(c^t ) = q^t$ for $0 \leq t \leq n-1$, be
the quasi-$YD$ datum from Example \ref{ex: basic example qYD datum}, let $R:= R((\Bbbk C_n, \omega_{\zeta }), c ,\chi )$,
and let $B:= R \# H$.    Suppose that there is a quasi-$YD$ datum for $B$, $((B,\omega_B), g_B, \chi_B)$ as in
Proposition \ref{pro:qydd}.   Then $g_B = 1_R \# c^w$ for some $0 \leq w \leq n-1$, and for $0 \leq t \leq n-1$
and $r \in R$,
\begin{equation}
   \chi _{B}\left( r\# c^{t}\right) = \varepsilon _{R}\left(
r\right) q^{-w t}\prod\limits_{0\leq i\leq t-1}    \omega_{\zeta }^{-1}\left(
c^{w}\otimes c^{i}\otimes c\right) = \varepsilon_R(r)q^{-wt}.  \label{form:chiA2short}
  \end{equation}
 In particular $g_{B}$ and $\chi _{B}$ are uniquely determined by $w$ and $%
\left( \left( \Bbbk   C_{n}  ,\omega_{\zeta }\right) ,c,\chi \right).
  $
  \end{proposition}

\begin{proof}
By Proposition \ref{pro:qydd} there exists $d = c^w
  $ such that $g_{B}=1_{R}\# d.$ Since $c^w \neq c  c^w$,
    Proposition \ref{pro:qydd}(ii) may be applied  to get that,
\begin{equation*}\chi_B(1_R \# c) = \chi(c^w)^{-1}   =
q^{-w}.
\end{equation*}
Since by Lemma \ref{lem:qYdOnSmash} $((\Bbbk C_n, \omega_{\zeta }),  c^w, \chi_B(1_R \# -))$ is a quasi-$YD$
datum, then  $\chi_B(1_R \# -)$ must satisfy \eqref{form:datum-1 short}, i.e.,
\begin{equation*}
\chi_B(1_R \# c^t) = \chi_B(1_R \# c)^t \prod\limits_{0\leq i\leq t-1}
 \omega_{\zeta}^{-1}(c^w \otimes c^i \otimes c)
= q^{-wt}\prod\limits_{0\leq i\leq t-1}  \omega_{\zeta}^{-1}(c^w \otimes c^i \otimes c).
\end{equation*}

The statement now follows from (i) of Proposition \ref{pro:qydd}.
\end{proof}

\begin{corollary}\label{cor: w=m}

Let $H,R,B$ be as in the proposition with $n=2m$.  If $((B,\omega_B), g_B= 1 \# c^w, \chi_B  )$ is a quasi-$YD$ datum
for $B$ with $c^w \neq 1$, then $w=m$.
\end{corollary}
\begin{proof}
By Lemma \ref{lem:qYdOnSmash}, if $((B,\omega_B), g_B= 1 \# c^w, \chi_B  )$ is a quasi-$YD$ datum, then
 $((\Bbbk C_{2m}, \omega_\zeta), c^w, \chi_B \sigma)$ is also a quasi-$YD$ datum where $\sigma$ is the inclusion map.
 Then by (\ref{form:datum-0 special}), $\chi_B \sigma(c)^n = \zeta^w$. By (\ref{form:chiA2short}),
 \begin{displaymath}
 \chi_B(1 \#c)^n = q^{-wn} = \zeta^{-w},
 \end{displaymath}
 so that $\zeta^{2w}=1$ and we must have that $w=m$.
\end{proof}

We are now able to construct a quasi-$YD$ datum on a dual quasi-bialgebra which is a bosonization of
a group algebra.  We begin with a useful lemma.

\begin{lemma}
\label{lem:go}Let $n,a\in
\mathbb{N}
$ with  $0\leq a\leq n^{2}-1.$ Then
\begin{equation*}
\left\vert \left\{ i\mid 0\leq i\leq a-1,i^{\prime }=n-1\right\} \right\vert
=\frac{a-a^{\prime }}{n}
\end{equation*}
\end{lemma}

 \begin{proof}
 Note that the left hand side of the equation above is the number of nonnegative
integers congruent to $n-1 \mod n$ and strictly less than $a$. For $t \geq 1$, define an interval $I_t$ of $n$ integers by
$I_t = \{ j \in \mathbb{N} | (t-1)n \leq j \leq tn-1  \}$. Then the left hand side is
the number of intervals $I_t$ whose entries are less than $a$. If $a = a^\prime + sn$, then $a \in I_{s+1}$ and
  this number is clearly $s$.
\end{proof}

In the next example, we find a quasi-$YD$ datum for $B:= R \# \Bbbk C_n$ where $n$ is even. As always, $q$ denotes a primitive $n^2$-rd root of unity
and $\zeta :=q^n$.

\begin{example}\label{ex: quasi YD for B} Let $n = 2m$  and let $(B:= R \# \Bbbk C_n, \omega_B = \omega_\zeta \pi)$ be the dual quasi-bialgebra of dimension $n^3$ constructed
via the quasi-$YD$ datum $((\Bbbk C_n, \omega_\zeta), c, \chi)$ for $q$ with $\chi(c^t) = q^t$, $0 \leq t \leq n-1$, as
in Example \ref{ex: basic example qYD datum}.  We will construct a quasi-$YD$ datum for $B$ for $\iota:= q^{-m^2}$,
a primitive $4$-th root of unity.

\par First note that $((\Bbbk C_n, \omega_\zeta), c^m, \widetilde{\chi})$ with $\widetilde{\chi}(c^t) = q^{-mt}$ for $0 \leq t \leq n-1$
is a quasi-$YD$ datum   by Proposition \ref{pro:speriamo} since $\widetilde{\chi}(c)^n
= q^{-mn} = \zeta^{-m} = \zeta^m$ since $\zeta$ has order $2m$.  Also $((\Bbbk C_n, \omega_\zeta), c^m, \widetilde{\chi})$ is a quasi-$YD$ datum for $\iota $ since
$\widetilde{\chi}(c^m) = q^{-m^2}$.

\par We will now show that $((B, \omega_B:= \omega_\zeta \pi^{ \otimes 3}), g_B:=\sigma(c^m), \chi_B:= \widetilde{\chi} \pi)$, where $\pi, \sigma$ are the
usual projection and inclusion maps from Remark \ref{rem: with a projection},  is a quasi-$YD$ datum.  Since $\pi$ is a surjection
of dual quasi-bialgebras from $(B, \omega_B)$ to $(\Bbbk C_n, \omega_\zeta)$ with $\pi \sigma (c^m) = c^m$, it remains to show that
for all $b \in B$,
\begin{displaymath}
\sigma(c^m) \widetilde{\chi} \pi(b_1)b_2 = b_1 \widetilde{\chi} \pi(b_2) \sigma(c^m)
\end{displaymath}
in order to apply Lemma \ref{lem:qYDquotient} and conclude that
$((B, \omega_B:= \omega_\zeta \pi), g_B:=\sigma(c^m), \chi_B:= \widetilde{\chi} \pi)$
is a quasi-$YD$ datum for $\iota$.

\par Let $b = \ x^{[a]} \# c^\ell$ for $0 \leq a \leq n^2 -1$ and $0 \leq \ell \leq n-1$.
Since by
   Theorem \ref{teo:RsmachHthin},
$
    \Delta_R(x^{[a]} )
=   \sum_{0 \leq i \leq a}\beta(i,a)  x^{[i]}
 \otimes  x^{[a-i]} $  and
 since   $\beta(0,a) = \beta(a,a)=1$, by  applying $\varepsilon_H$ on the left and on the right of  (\ref{form:piIdpi}),
 \begin{equation}
      \pi(b_1) \otimes b_2 =  c^{a}c^\ell \otimes x^{[a]} \# c^\ell = c^a c^\ell \otimes b
 \text{ and }
 b_1
   \otimes \pi(b_2) =   x^{[a]} \# c^{\ell} \otimes c^\ell = b \otimes c^\ell .
 \end{equation}

By the formula for multiplication in $B = R \# \Bbbk C_n$ in Theorem \ref{teo:RsmashH}, we have that
\begin{equation*}
(1 \# c^m)(x^{[a]} \# c^\ell) = \omega_\zeta(c^m \otimes c^{a} \otimes c^\ell) \omega_\zeta^{-1}(c^{a} \otimes c^m \otimes c^\ell)
c^m \rhd x^{[a]} \# c^m c^\ell,
\end{equation*}
and
\begin{equation*}
(x^{[a]} \# c^\ell)(1 \# c^m) =
 \omega_\zeta^{-1}(c^a \otimes c^\ell \otimes c^m) x^{[a]} \# c^\ell c^m
=
 \omega_\zeta^{-1}(c^a \otimes c^m \otimes c^\ell) x^{[a]} \# c^\ell c^m.
\end{equation*}

Since $c^m \rhd x^{[a]} = \chi_{[a]}(c^m)x^{[a]}$ where $\chi_{[a]}$ is defined
 in  Proposition \ref{prop: structure of T(V)}, it remains to show that
\begin{equation*}
\omega_\zeta(c^m \otimes c^{a} \otimes c^\ell)
\chi_{[a]}(c^m) \widetilde{\chi}(c^{a}c^{\ell})
= \widetilde{\chi}(c^\ell).
\end{equation*}

By (\ref{form:chiProd}) for the quasi-$YD$ datum $\left( (\Bbbk C_n, \omega_\zeta),c^{m},\widetilde{\chi }\right) $,
\begin{displaymath}
\widetilde{\chi }\left( c^{a}c^{\ell}\right)
 =  \omega_\zeta^{-1}(c^{m}\otimes c^{a}\otimes c^{\ell}) \widetilde{\chi }(c^a) \widetilde{\chi}(c^\ell),
\end{displaymath}

and
thus it suffices to prove that
$
\widetilde{\chi }\left( c^{a}\right) \chi _{\left[ a\right] }\left(
c^{m}\right) =1$. Since  $ \widetilde{\chi }\left( c^{a}\right) =
\widetilde{\chi }\left( c^{a^\prime}\right) = q^{-ma'}$, this is equivalent to  showing that $$
\chi _{\left[ a\right] }\left( c^{m}\right) =  q^{ma^{\prime }}.
$$
Since $c^m$ is a cocommutative element, by equation (\ref{form:chincocomm})
\begin{eqnarray*}
\chi _{\left[ a\right] }\left( c^{m}\right) &=&
 \chi(c^m)^a \prod\limits_{0\leq i\leq a-1} \omega_\zeta(c^m \otimes c \otimes c^i)
 = q^{ma}\prod\limits_{0\leq i\leq a-1} \omega_\zeta(c^m \otimes c \otimes c^{i^\prime})\\
 &=&   q^{ma}\prod\limits_{0\leq i\leq a-1}\zeta ^{m [[1+i^\prime \geq n]] } =
 q^{ma}\prod\limits_{0\leq i\leq a-1}\zeta ^{m\delta _{i^{\prime },n-1}}.
\end{eqnarray*}

Thus we have to prove that%
\begin{equation*}
q^{ma}\prod\limits_{0\leq i\leq a-1}\zeta ^{m\delta _{i^{\prime
},n-1}}=q^{ma^{\prime }}.
\end{equation*}%
But   $q^{ms} = q^{-ms}$ for every $s\in n\mathbb{Z}$ since,  writing $s=n\widehat{s}$, $q^{2mn\widehat{s}} = q^{n^2\widehat{s}} = 1$.
Thus it suffices to prove that
\begin{equation*}
\prod\limits_{0\leq i\leq a-1}\zeta ^{m\delta _{i^{\prime
},n-1}}=q^{-m\left( a-a^{\prime }\right)} = q^{m\left( a-a^{\prime }\right)  } .
\end{equation*}
By Lemma \ref{lem:go}, we have
\begin{equation*}
\left\vert \left\{ i\mid 0\leq i\leq a-1,i^{\prime }=n-1\right\} \right\vert
=\frac{a-a^{\prime }}{n},
\end{equation*}%
so that
   \begin{equation*}
\prod\limits_{0\leq i\leq a-1}\zeta ^{m\delta _{i^{\prime
},n-1}}
=\prod\limits_{0\leq i\leq a-1}q^{nm\delta _{i^{\prime
},n-1}}
=q^{mn\sum_{0\leq i\leq a-1}\delta _{i^{\prime },n-1}}=q^{mn\frac{%
a-a^{\prime }}{n}}
=q^{m\left( a-a^{\prime }\right) }.
\end{equation*}
This shows that $((B, \omega_B), g_B:= \sigma(c^m), \chi_B:= \widetilde{\chi} \pi)$ is a quasi-$YD$ datum for $q^{-m^2}$ and thus
one can form the bosonization $S \# B$ of dimension $4n^3= 32 m^3$ where $S$ has basis $y^{[i]}$ for $0 \leq i \leq 3$.
\qed

\end{example}

\end{document}